\documentclass{amsproc}

\usepackage{a4wide}
\usepackage{amsmath}
\usepackage{enumerate}

\usepackage{amsmath,amsthm}
\usepackage{amsfonts}
\usepackage{amssymb}
\usepackage{longtable}
\usepackage[matrix,arrow,curve]{xy}

\usepackage{amsmath}
\usepackage{amsfonts}
\usepackage{amssymb}

\newtheorem{rem}{Remark}{\bf}{\rm}{\rm}

\newtheorem*{theorem*}{Theorem}
\newtheorem*{cor*}{Corollary}

\newtheorem{theorem}{Theorem}
\newtheorem{cor}{Corollary}
\newtheorem{prop}{Proposition}

\newtheorem{lem}{Lemma}

\newtheorem{ex}{Example}{\rm}{\rm}

\def\Real{\mathbb{R}}
\def\SO{\text{\rm SO}}
\def\SU{\text{\rm SU}}
\def\SL{\text{\rm SL}}

\newcommand{\frakg}{\mathfrak{g}}
\newcommand{\frakh}{\mathfrak{h}}
\newcommand{\frakf}{\mathfrak{f}}
\newcommand{\mathR}{\mathbb{R}}
\newcommand{\frakso}{\mathfrak{so}}

\newcommand{\fb}{\mathfrak{b}}
\newcommand{\n}{\mathfrak{n}}

\newcommand{\frakk}{\mathfrak{k}}
\newcommand{\g}{\mathfrak{g}}
\newcommand{\f}{\mathfrak{f}}
\newcommand{\fu}{\mathfrak{u}}
\newcommand{\su}{\mathfrak{su}}
\newcommand{\m}{\mathfrak{m}}
\newcommand{\h}{\mathfrak{h}}
\newcommand{\so}{\mathfrak{so}}

\newcommand{\hol}{\mathfrak{hol}}

\newcommand{\zr}{\ltimes}

\newcommand{\Su}{\mathfrak{S}}

\def\im{\mathop\text{\rm im}\nolimits}

\def\pr{\mathop\text{\rm pr}\nolimits}

\def\R{\mathcal{R}}

\newcommand{\frakc}{\mathfrak{c}}
\newcommand{\fraka}{\mathfrak{a}}

\usepackage{stackrel}

\newcommand{\be}{\begin{equation}}
\newcommand{\ee}{\end{equation}}

\let\leq=\leqslant
\let\geq=\geqslant

\begin{document}
	
	\title{On Lorentzian connections with parallel skew torsion}

	\author{Igor Ernst}\thanks{$^1$Department of Mathematics and Statistics, Masaryk University, Faculty of Science, Kotl\'a\v{r}sk\'a 2, 611 37 Brno, Czech
		Republic}
	
	\author{Anton S. Galaev}\thanks{$^2$University of Hradec Kr\'alov\'e, Faculty of Science, Rokitansk\'eho 62, 500~03 Hradec Kr\'alov\'e,  Czech
		Republic\\
		E-mail: anton.galaev(at)uhk.cz}

	\begin{abstract}
		The paper is devoted to metric connections with parallel skew-symmetric torsion in Lorentzian signature. This is motivated by recent progress in the Riemannian signature and by possible applications to  supergravity theories. 
		We provide a complete information about holonomy algebras, torsion and curvature of the considered connections up to the corresponding objects from the Riemannian signature. Various examples are constructed. It is shown how to construct all simply connected Lorentzian  naturally reductive homogeneous spaces of arbitrary dimension from Riemannian naturally reductive homogeneous spaces. This leads to complete classification of Lorentzian  naturally reductive homogeneous spaces in low dimensions.
		
		\vskip0.5cm
		
		{\bf Keywords}: Lorentzian manifold; parallel skew-symmetric torsion; holonomy; naturally reductive homogeneous space; type II supergravity.
		\vskip0.5cm
		
		
		
	\end{abstract}

	\maketitle

	\tableofcontents

\section{Introduction}

During the last decades a careful attention of many researchers was taken to Riemannian connections with torsion and related geometric structures, in particular
the holonomy,  e.g., \cite{AF04,A10,AFe14,AFeF15}. An important   
special case is provided by connections with parallel skew-symmetric torsion, these include, e.g., naturally reductive homogeneous spaces, Sasakian and 3-Sasakian manifolds, nearly K\"ahler manifolds and some other. 
Riemannian naturally reductive homogeneous spaces are classified  up to dimension 8 \cite{TV,KV,AFeF15,Storm2}. Irreducible holonomy groups of Riemannian metric connections with parallel skew-symmetric torsion were studied in \cite{CS04}. The general case  cannot be reduced to the case of irreducible holonomy and this situation was studied very recently in \cite{CMS21}.

Metric connections with skew-symmetric torsion play an important role in  mathematical physics (string theory, supergravity theory), see \cite{FF09,FI02,MSh} and the references therein. A sector of the Killing spinor and field equations of certain supergravity theories in dimensions 6 and 10 may be written in the form
$$ \nabla\Psi=0,\quad (H-2d\Phi)\cdot\Psi=0,\quad \textrm{Ric}^\nabla+2\nabla^gd\Phi=0,$$
where  $\nabla$ is a Lorentzian metric connection with skew-symmetric torsion $H$, $\Psi$ is a spinor field, and $\Phi$ is a function.  
The holonomy group of the connection $\nabla$ controls parallel object and the Ricci tensor, hence it may provide information about the solutions. In type II string theory one investigates manifolds $N^k \times M^{10 - k}$, where $N^k$ is a $k$-dimensional space-time and $M^{10 - k}$ is a Riemannian manifolds equipped with a 3-form $H$ defining the connection $\nabla$.  
 It is also interesting to consider more general backgrounds with the 3-form $H$ defined on the entire Lorentzian manifold, this would give new interesting solutions exploiting the Lorentzian nature of the backgrounds~\cite{FF09,Str}. 

In the present paper we consider Lorentzian manifolds $(M,g)$ with metric connections $\nabla$ that have parallel skew-symmetric torsion $T$. First in Section \ref{Secpomega} we show that if the holonomy algebra of $\nabla$ is weakly irreducible, i.e., it does not preserve any non-degenerate subspace of the tangent space, and  $M$ is simply connected, then  there exists a $\nabla$-parallel isotropic vector field $p$, and the torsion is of the form $p\wedge\omega$, where $\omega$ is a parallel bivector on the screen bundle $p^\bot/\left<p\right>.$ 
Then we describe general structures carrying a parallel isotropic vector fields and having the just indicated torsion.
 It turns out that the vector field $p$ is also parallel with respect to the Levi-Civita connection $\nabla^g$. Thus the holonomy algebras of the connections $\nabla^g$ and $\nabla$ are contained in the Lie algebra $\mathfrak{so}(n)\zr\mathbb{R}^n$. We show that the projections of these algebras
to $\mathfrak{so}(n)$ coincide and provide examples of structures such that the projections of the holonomy algebras
to $\mathbb{R}^n$ do not coincide. Examples of the spaces considered here include known flat Lorentzian connections with closed torsion as well as regular homogeneous plane waves. Then in Section \ref{SecRed} we start to assume that the holonomy algebra $\g$ preserves an orthogonal decomposition $L\oplus E$ of the tangent space and the induced representation of $\mathfrak{g}$ in $L$ is weakly irreducible. In Sections \ref{Secvert}, \ref{SecdimL>3}, \ref{SecL=2,3} we give a detailed description of the holonomy algebras and the corresponding torsion and curvature. We provide various examples. We use the results from \cite{CMS21} to show that in some cases the manifold is foliated by manifolds of smaller dimension and locally on the leaf space one obtains an induced metric connection with parallel skew-symmetric torsion. At the same time, we show that \emph{not} all statements proved in \cite{CMS21} for Riemannian manifolds are valid for Lorentzian manifolds.

In Section \ref{SecHom} 
 we study simply connected Lorentzian  naturally reductive homogeneous spaces. Previously such spaces were classified in dimension 3, see \cite{CM08,HBRR}, and 4, see~\cite{BLM15,CLBook}. Using the above results we show how to construct all simply connected Lorentzian  naturally reductive homogeneous spaces of arbitrary dimension from Riemannian naturally reductive homogeneous spaces. We consider in details the spaces of  dimension 3, 4, and 5. 

\vskip0.5cm

{\bf Acknowledgements.} The authors are thankful to Ilka Agricola, Ioannis Chrysikos,  Carlos Shahbazi for useful discussions and to the anonymous referee for valuable comments and suggestions that have significantly improved the article.  This work was supported from Operational Programme Research, Development and Education – “Project Internal Agency of Masaryk University” (No. CZ.02.2.69/0.0/0.0/19\_073/0016943). A.G. acknowledges institutional support of University of Hradec Kr\'alov\'e.

\section{Preliminaries} \label{SecPrel}
Let us collect basic information about metric connections with parallel skew-torsion, which may be found, e.g., in \cite{AF04,A10,AFeF15,CMS21}.
Consider a triple $(M, g, \nabla)$, where $g$ is a pseudo-Riemannian metric on a smooth manifold $M$, and $\nabla$ is a metric connection on $M$ with parallel skew-symmetric torsion $T$, i.e., it holds that
$$\nabla g=0,\quad \nabla T=0,\quad g(T(X,Y),Z)=-g(T(X,Z),Y).$$
The connection $\nabla$ is related to the Levi-Civita connection $\nabla^g$ by the equality
$$ \nabla = \nabla^g + \frac{1}{2}T. $$
Following \cite{CMS21} we call the triple  $(M, g, \nabla)$ a geometry with parallel skew-symmetric torsion. We will denote such geometry also 
by $(M, g, T)$. We will consider $T$ as a tensor field of different types. By abuse of notation we will write
$$ T(X, Y, Z) = g(T(X, Y), Z)=g(T(X)Y,Z).$$ Similar notation will be used for all differential $k$-forms.
We will identify bivectors with skew-symmetric endomorphisms using the equality $$(X\wedge Y)Z=g(X,Z)Y-g(Y,Z)X.$$

The curvature tensors of the connection $\nabla$ and the Levi-Civita connection are related by the formula
\begin{equation}
\label{RRg} R(X,Y)=R^g(X,Y)+\frac{1}{4}[T(X),T(Y)]-\frac{1}{2}T(T(X)Y).
\end{equation} 
This implies
\begin{equation}
\label{sympropR} g(R(X,Y)Z,W)=g(R(Z,W)X,Y).
\end{equation} 

The equality
$$\sigma_T(X,Y,Z)=\underset{X Y Z}{\Su}T(T(X,Y),Z)$$
defines a 4-form $\sigma_T$. It holds also that
\begin{equation}\label{T(X)T}
(T(X)\cdot T)(Y,Z,V)=\sigma_T(Y,Z,V,X).
\end{equation}

The condition $\nabla T=0$ implies
\begin{equation}\label{dT}
dT=2\sigma_T,\quad \nabla^g T=\frac{1}{2}\sigma_T,
\end{equation}
and the first Bianchi identity may be written in the form
\begin{equation}\label{Bident}\underset{X Y Z}{\Su}R(X,Y)Z=\sigma_T(X,Y,Z).\end{equation}

Let $(M, g, \nabla)$ be a geometry with parallel skew-symmetric torsion.
Let $\g$ be the holonomy algebra of the connection $\nabla$ at a point $x\in M$. 
We identify the tangent space $ (T_xM, g_x) $ with the pseudo-Euclidean  space $\big(\mathR^{r, s}, (\cdot,\cdot)\big)$.
Then $\g$ is identified with a subalgebra of $\so(r,s)$.

Let $\g$ be an arbitrary subalgebra of $\so(r, s)$. Let $T\in\wedge^2\Real^{r,s}\,\otimes\Real^{r,s}$ be a tensor.
The space of curvature tensors $\mathcal{R}^T(\g)$ with  torsion $T$ is defined as the space of maps $R \in \wedge^2 \mathR^{r, s} \otimes \g$ satisfying the first Bianchi identity~\eqref{Bident}.
If $T=0$, we denote the space $\mathcal{R}^T(\g)$ just by $\mathcal{R}(\g)$.
Let $\mathcal{L}(\mathcal{R}^T(\g))$ be the vector subspace in $\g$ spanned by images of the elements from $\mathcal{R}^T(\g)$. We call a subalgebra $\g\subset\so(r, s)$ a \emph{Berger algebra with torsion} $T$ if 
$\mathcal{L}(\mathcal{R}^T(\g))=\g$. 

\begin{prop}
	Let $(M,g)$ be a pseudo-Riemannian manifold and $\nabla$ a  metric connection on $(M,g)$ with parallel torsion (not necessary skew-symmetric).
	If $\g \subset \so(r,s)$ is the holonomy algebra of the connection $\nabla$ at a point $x\in M$, then $ \g $ is a Berger algebra with the torsion $T_x$.
\end{prop}
\begin{proof}
	Denote the curvature tensor of the connection $\nabla$ by $R$.
	By the Ambrose-Singer theorem on holonomy, the holonomy algebra $\g$ of the connection $\nabla$ at the point $x\in M$ is spanned by the endomorphisms of $T_xM$ of the form
	$$  \tau^{-1}_\gamma R_y(X,Y) \tau_\gamma,$$
	for all piecewise smooth curves  $\gamma$ starting at the point $x$ and all vectors $X,Y \in T_yM$, where $y$ is the end-point of the curve $\gamma$. Here $\tau_\gamma$ denotes the parallel transport along the curve $\gamma$.
	Since $\nabla T=0$, it holds  that $\tau^{-1}_\gamma T_y(\tau_\gamma X,\tau_\gamma Y) = T_x(X,Y)$. Using this and the Bianchi identity we get
	\begin{multline*}
	\underset{X Y Z}{\Su} \tau^{-1}_\gamma R_y(\tau_\gamma X,\tau_\gamma Y)\tau_\gamma Z = 
	\underset{X Y Z}{\Su} \tau^{-1}_\gamma T_y(T_y(\tau_\gamma X,\tau_\gamma Y),\tau_\gamma Z) = \\
	= \underset{X Y Z}{\Su} \tau^{-1}_\gamma T_y(\tau_\gamma T_x(X,Y),\tau_\gamma Z) = \underset{X Y Z}{\Su} T_x(T_x(X,Y),Z).
	\end{multline*}
	This shows that 
	the map defined by 
	$$ (X,Y) \mapsto \tau^{-1}_\gamma R_y(\tau_\gamma X,\tau_\gamma Y)\tau_\gamma $$ 
	is an element of $ \mathcal{R}^{T_x}(\g)$.
	This completes the proof.
\end{proof}

Note that if $\mathcal{R}^{T}(\g)$ is non-empty, then it is an affine space with the corresponding vector space $ \mathcal{R}(\g)$. In particular, if
$ \mathcal{R}(\g)=0,$ then $ \mathcal{R}^{T}(\g)$ is either empty or contains only one element.
The following statement is known \cite[Lemma 5.6]{CS04},  we prove it for completeness.

\begin{cor} In the settings of the proposition, if $ \mathcal{R}(\g)=0$,
	i.e., the space $ \mathcal{R}^{T_x}(\g)$ contains only one element, then $\nabla R=0$.
\end{cor}
\begin{proof}
	The fact that $R_x$ as well as all elements $\tau^{-1}_\gamma R_y(X,Y) \tau_\gamma$ belong to $\mathcal{R}^{T_x}(\g)$ implies that
	$\tau^{-1}_\gamma R_y(X,Y) \tau_\gamma=R_x$, i.e., $R$ is $\nabla$-parallel.
\end{proof}

\medskip

An important class of geometries with parallel skew-symmetric torsion form naturally reductive pseudo-Riemannian homogeneous spaces. Here we are following the exposition from the book~\cite{CLBook}. By the Ambrose-Singer theorem for homogeneous structures, a connected simply connected complete pseudo-Riemannian manifold $(M,g)$ is a reductive homogeneous manifold if and only if 
$(M,g)$ admits a connection $\nabla$ such that it holds
$$\nabla g=0,\quad \nabla R=0,\quad \nabla T=0.$$ Such a connection $\nabla$ is called an Ambrose-Singer connection (AS-connection).
 Suppose that $(M,g)$ is reductive homogeneous and fix an AS-connection $\nabla$.
Fix a point $o\in M$ and denote by $\m$ the tangent space $T_oM$. Let $\g\subset\so(\m)$ be the holonomy algebra of the connection $\nabla$ at the point $o$. Since the curvature tensor of the connection $\nabla$ is parallel, by the Ambrose-Singer theorem on holonomy,
$$\g=\im R_o=R_o(\m,\m)=\textrm{span}\{R_o(X,Y)|X,Y\in\m\}.$$
Since $R$ and $T$ are $\nabla$-parallel, it holds
\begin{equation}
\g\cdot R_o=\g\cdot T_o=0.\label{propRT1}
\end{equation}
The first and the second Bianchi identities take the form
\begin{align}
\underset{X Y Z}{\Su}R_o(X,Y)Z&=\underset{X Y Z}{\Su}T_o(T_o(X,Y),Z),\label{propRT2}\\
\underset{X Y Z}{\Su}R_o(T_o(X,Y),Z)&=0\label{propRT3}.
\end{align}
Consider the vector space \begin{equation}\label{reddecf} \f=\g\oplus\m.\end{equation} The above equalities show that the assignment
\begin{align*}[A,B]&=[A,B]_\g,\quad [A,X]=AX,\\ [X,Y]&=-R_o(X,Y)-T_o(X,Y),\quad A,B\in\g,\,\,X,Y\in\m\end{align*} 
defines the Lie bracket on $\f$ (here $[\cdot,\cdot]_\g$ is the Lie bracket of the Lie algebra $\g$). It is clear that the decomposition \eqref{reddecf} is reductive. 
The Lie algebra $\g$ is called \emph{the transvection algebra} of the structure $(M,g,\nabla)$.

Suppose now that a Lie algebra $\f$ with a reductive decomposition \eqref{reddecf} is given. 
Let $F$ be the simply connected Lie group with the Lie algebra $\f$.  
Let $G\subset F$ be the connected Lie subgroup corresponding to the subalgebra $\g\subset\f$. The reductive decomposition \eqref{reddecf} is called \emph{regular} if the subgroup  $G\subset F$ is closed. This is the case if  $\f$ is a transvection algebra. 
Suppose that the reductive decomposition \eqref{reddecf} is regular and it holds $[\m,\m]_\g=\g$, where $$[\m,\m]_\g=\textrm{span}\{[X,Y]_\g|X,Y\in\m\},$$
and
$[\cdot,\cdot]_\g$ denotes the projection of the Lie bracket to $\g$.  Suppose that the representation of $\g$ in $\m$ is faithful, and there is a $\g$-invariant pseudo-Euclidean metric on the vector space $\m$.
These data define the reductive homogeneous space $(F/G,g)$.  
Let $\nabla$ be the canonical connection on $(F/G,g)$. The curvature tensor and the torsion of $\nabla$ are parallel with respect to $\nabla$. The tangent space at the origin $o\in F/G$ is identified with $\m$ and it holds 
\begin{equation}
\label{RTbra}
R_o(X,Y)=-[X,Y]_\g,\quad T_o(X,Y)=-[X,Y]_\m,\quad X,Y\in \m.\end{equation} 
Under that construction, if $\f$ is the transvection algebra of a structure $(M,g,\nabla)$, then  the just constructed structure $(F/G,g,\nabla)$ is canonically isomorphic to the initial $(M,g,\nabla)$. The triple $(\m,R_o,T_o)$ is called \emph{the infinitesimal model} of the transvection algebra $\f$ (or of the corresponding reductive homogeneous space).

Finally, a reductive pseudo-Riemannian homogeneous space $(F/G,g)$ with a reductive decomposition~\eqref{reddecf} is called \emph{naturally reductive} if 
 it holds that 
$$g_o([X,Y]_\m,Z)=-g_o([X,Z]_\m,Y),\quad \forall X,Y,Z\in\m,$$ where $[X,Y]_\m$ is the projection of $[X,Y]$ to $\m$. The  AS-connections corresponding to naturally reductive homogeneous spaces are exactly these with skew-symmetric torsion.

\section{Connections with parallel isotropic vector field $p$ and torsion~$p\wedge \omega$} \label{Secpomega}

In this section we show that if the manifold $M$ is simply connected and the holonomy algebra of a structure $(M,g,\nabla)$  is weakly irreducible, then on $M$ there exists a $\nabla$-parallel isotropic vector field $p$, and the torsion is of the from $p\wedge \omega$ for a certain bivector $\omega$. Then we study general structures $(M,g,\nabla)$  admitting $\nabla$-parallel isotropic vector fields $p$ and having the torsion $p\wedge \omega$.

Recall that a subalgebra of a pseudo-orthogonal algebra is called  weakly irreducible if it does not preserve any proper non-degenerate subspace of the pseudo-Euclidean space. For a subalgebra of the orthogonal algebra  this condition is equivalent to the irreducibility. Irreducible holonomy algebras of metric  connections with parallel skew torsion in the Riemannian signature were studied in~\cite{CS04}. 

Let us recall the classification of weakly irreducible subalgebras $\g\subset\so(1,n+1)$, $n\geq 1$ \cite{BBI}. If $\g$ is irreducible, then   $\g=\so(1,n+1)$. It is clear that if $\g=\so(1,n+1)$ is the holonomy algebra of a connection with non-zero parallel skew torsion, then $n=1$, and the torsion form is proportional to the volume form. By this reason we assume that $\g$ is weakly irreducible and not irreducible. In this case~$\g$ preserves an isotropic line in $\Real^{1,n+1}$. 
Fix a Witt basis $p,e_1,\dots,e_n,q$ of the Minkowski space $\Real^{1,n+1}$. We will denote by $\Real^n$ the vector subspace of $\Real^{1,n+1}$ spanned by the vectors $e_1,\dots,e_n$. With respect to the basis $p,e_1,\dots,e_n,q$ the subalgebra of $\so(1,n+1)$ preserving the isotropic line $\Real p$ has the following matrix form:
$$
\so (1, n + 1)_{\Real p} = \left\{
\left. 
\begin{pmatrix}     
a & -X^t & 0 \\
0 & A & X \\
0 & 0 & -a
\end{pmatrix} \right|
\begin{matrix}     
a \in \mathR \\
A \in \frakso (n) \\
X \in \mathR^n
\end{matrix} \right\} .
$$	
The above matrix may be identified with the bivector
$$-ap\wedge q+A+p\wedge X,$$
and we get the decomposition 
$$\so (1, n + 1)_{\Real p} = (\mathR p\wedge q \oplus  \so(n))\zr p \wedge \mathR^n.$$
There are four types of weakly irreducible  subalgebras  $\g\subset\so(1, n + 1)_{\Real p}$:
\begin{itemize}
	\item[Type 1.] $ \g^{1, \h} = (\mathR p\wedge q \oplus\h )\zr p \wedge \mathR^n$,
	\item[Type 2.] $\g^{2, \h} =\h\zr p \wedge \mathR^n,$
	\item[Type 3.] $\g^{3, \h, \varphi} =  \{ \varphi(A) p\wedge q + A \, | \, A \in \h \}\zr p \wedge \mathR^n,$ 
	\item[Type 4.] $\g^{4, \h, m, \psi} =  \{ A+p\wedge \psi(A)  \, | \, A \in \h \}\zr p \wedge \mathR^m.$ 
\end{itemize} 
Here $\h \subset \so(n)$  is a subalgebra, $\varphi : \h \to \Real$  is a non-zero linear map with the property $\varphi|_{[\h, \h]} = 0$. For the last algebra there exists an orthogonal decomposition  $\mathR^n = \mathR^m \oplus \mathR^{n - m}$  such that $\h \subset \so(m)$, and $\psi:\h\to \mathR^{n - m}$  is a surjective linear map with  $\psi|_{[\h, \h]} = 0$. The subalgebra $\h \subset \so(n)$ coincides with the $\so(n)$-projection of $\g$ and it is called the orthogonal part of $\g$.

Let $(M, g, \nabla)$ be a Lorentzian geometry with non-zero parallel skew-symmetric torsion $T$, $\dim M=3$ and suppose that the holonomy algebra of $\nabla$ is weakly irreducible. Then $T$ is proportional to the volume form, and $\g\subset\so(1,2)$ is one of the Lie algebras
$\so(1,2)$, $\Real p\wedge q\zr\Real p\wedge e_1$, $\Real p\wedge e_1$. Thus we may assume that $\dim M\geq 4$.

\begin{lem}\label{lemwirr}
	Let $\g\subset\so(1,n+1)_{\Real p}$ be a weakly irreducible subalgebra, where $n\geq 2$. Suppose that $\g$ annihilates a non-zero 3-vector $S \in \wedge^3 \mathR^{1, n + 1}$. Then $\g$ is either of type 2 or 4, i.e., it annihilates the isotropic vector $p$. Moreover, it holds $$S=p\wedge \omega,$$ where $\omega\in\wedge^2\Real^n$, and the subalgebra $\h\subset\so(n)$ annihilates $\omega$. 
\end{lem}

\begin{proof}
Consider as above a Witt basis $p, e_1, \dots, e_n, q$ of the   Minkowski space $ \mathR^{1, n + 1}$. Any 3-vector $S \in \wedge^3 \mathR^{1, n + 1}$ may be decomposed as
	\begin{equation}
	\label{decomposiotionOfThreeForm}
	S = p\wedge q \wedge X + p\wedge \omega + q \wedge \eta + \xi,
	\end{equation} 
	where $X \in \mathR^n, \, \omega, \eta \in \wedge^2 \mathR^n, \, \xi \in \wedge^3 \mathR^n.$ For weakly irreducible algebras of each of the four types we will consider the equation 
	$$ \g \cdot S = 0. $$
	
	First consider 3-vectors that are annihilated by $p \wedge V,$ where $V \in \mathR^n$ is non-zero: 
	\begin{equation}
	\label{invariantTorsionElement}
	(p\wedge V) \cdot S = p \wedge V \wedge X + V \wedge \eta + p \wedge q \wedge \eta(V) - p\wedge \xi(V) = 0. 
	\end{equation}
	This is equivalent to the following system of equations:
	$$ V \wedge X  = \xi(V), \quad  \eta(V)  =  0,\quad  
	V \wedge \eta  = 0.
	$$
	Substituting $V$ to the third equation, we get
	$$ 0 = (V \wedge \eta) (V) = (V, V) \eta = 0. $$  
	This implies $\eta = 0.$ Substituting $V$ to the first equation we get
	$ (V, V) X = 0. $
	Thus, $X = 0 $ and the only remaining equation is $ \xi(V) = 0.$
	This equation  holds true if and only if $ \xi \in \wedge^3 (V^\perp).$
	Concluding, a 3-vector $S$ on $\mathR^{1,n + 1}$ is annihilated by an element $p\wedge V,$ where $V \in \mathR^n$ if and only if $S$ has the form
	$$ S = p\wedge \omega + \xi, $$
	where $\omega \in \wedge^2 \mathR^n, \, \xi \in \wedge^3 (V^\perp). $
	Therefore, a 3-vector $S$ is annihilated by $\g^{2, \h}$ or $\g^{4, \h,\psi}$ only if $S = p \wedge \omega,$ where $\omega \in \wedge^2 \mathR^n$ is annihilated by $\h$. Next, it holds that 
	$$ (p\wedge q ) \cdot (p\wedge \omega) = -p\wedge w. $$
	Therefore, there are no non-zero 3-vectors annihilated by the algebras of types 1 and 3. 
\end{proof}


Let $(M, g, \nabla)$ be a Lorentzian geometry with non-zero parallel skew-symmetric torsion $T$. Suppose that there exists a $\nabla$-parallel isotropic vector field $p$. Since the metric $g$ is $\nabla$-parallel, the distribution $p^\bot$ is $\nabla$-parallel.
The bundle $E=p^\bot/\left<p\right>$ is called  \emph{the screen bundle}, see, e.g., \cite{LScr}. There is the obvious projection $p^\bot\to E$. The connection $\nabla$ induces a connection $\nabla^E$ on $E$: if $X$ is a vector field on $M$ and $Y$ is a section of $E$, then  $\nabla^E_XY$ is the projection to $E$ of the vector field $\nabla_X \tilde Y$, where  $\tilde Y$ is an arbitrary section of $p^\bot$ such that its projection to $E$ is $Y$.

\begin{cor}\label{corwirrcase}
	Let $(M, g, \nabla)$ be a Lorentzian geometry with non-zero parallel skew-symmetric torsion $T$, $\dim M = n + 2\geq 4$. Let $\g$ be the holonomy algebra of $\nabla$. Suppose that $\g$ is weakly irreducible and suppose that $M$ is simply connected. Then on $M$  there exists a $\nabla$-parallel isotropic vector field $p$ and a $\nabla^E$-parallel bivector $\omega$ on the screen bundle $E=p^\bot/\left<p\right>$ such that for $T$ it holds $$T=p\wedge \omega.$$  
\end{cor}

Motivated by the corollary, in the rest of this section we consider
 a Lorentzian geometry $(M, g, \nabla)$ carrying a $\nabla$-parallel isotropic vector field $p$  such that   the torsion is given by a 3-vector  
$$ T = p \wedge \omega, $$
where $\omega \in \wedge^2 E$ is a $\nabla^E$-parallel bivector on the screen bundle $E= p^\perp/\left<p\right>$. 

It holds that $$ T(X, Y) = g(p,X)\omega( Y) - g(p,Y) \omega( X) + \omega(X,Y)p $$
for all vector fields $X$ and $Y$ on $M$.

The curvature tensor of $\nabla$ takes the form 
\begin{equation}
\label{curvatureTensorOmega} 
R(X, Y) = R^g(X,Y) + \frac{1}{4} p \wedge \big(g(p, Y) \omega^2( X) - g(p, X) \omega^2( Y) \big) . 
\end{equation}

From \eqref{T(X)T} it follows that
$$\sigma_T=0.$$ 
This and the Bianchi identity imply that the holonomy algebra $\g$ of the connection $\nabla$ is a Berger algebra with zero torsion, and consequently $\g$ is the holonomy algebra of the Levi-Civita connection on a Lorentzian manifold, see, e.g., \cite{GalRMS}. From \eqref{T(X)T} it follows also that
 $T$ is parallel with respect to the Levi-Civita connection:
$$   \nabla^g T=0.$$
From \eqref{dT} it follows that
$$dT=0.$$ It can be shown that conversely each indecomposable Lorentzian geometry with closed torsion is of the type considered in this section. 

Next, the vector field $p$ is parallel with respect to $\nabla^g:$
$$ \nabla^g_X p = \nabla_X p - \frac{1}{2}T(X, p) = 0.$$
This shows that $(M, g)$ is a Walker manifold with a parallel isotropic vector field. Therefore, locally there are coordinates $v, x^1, \dots, x^n, u$ such that the metric $g$  takes the form 
\begin{equation} \label{walkerMetric}
g = 2 dv du + h + 2A du + H (du)^2, 
\end{equation}
where $h=\sum_{i,j=1}^nh_{ij}(x^1, \dots, x^n, u) dx^i dx^j$ is a $u$-family of local Riemannian metrics,\\ $A=\sum_{i = 1}^n A_i(x^1, \dots, x^n, u) dx^i$ is a 1-form, and $H=H(x^1, \dots, x^n, u)$ is a local function (see, e.g., \cite{Walkerbook,GalRMS}). The vector field $\partial_v$ is isotropic and parallel. The 2-form $\omega$ may be expressed as
$$\omega=2\sum_{1\leq i<j\leq n}\omega_{ij} dx^i\wedge dx^j.$$

We see that both the holonomy algebras of the connections $\nabla$ and $\nabla^g$ are contained in $\so(n)\zr\Real^n\subset\so(1,n+1)_{\mathbb{R} p}$. Now we are going to compare these holonomy algebras. 

Consider the screen bundle
$$ E = p^\perp / \left<p \right>.$$
Note that the bundle $ p^\perp$ is parallel with respect to both $\nabla$ and $\nabla^g.$ Denote by $\nabla^{g,E}$ the connection on $E$ induced by $\nabla^g$. The 2-form $\omega$ on $E$ is parallel with respect to the both connections on $E$.
It holds that $$\nabla^E_X=\nabla_X^{g,E},\quad\forall X\in\Gamma(p^\bot),\quad \nabla^E_{\partial_u}=\nabla_{\partial_u}^{g,E}+\frac{1}{2}\omega.$$ 
The holonomy algebra $\hol_x(\nabla^{g,E})$ coincides with the $\so(n)$-projection of the holonomy algebra of the connection $\nabla^{g}$ \cite{LScr}. Likewise, it is not hard to check that $\hol_x(\nabla^{E})$ coincides with the $\so(n)$-projection of $\g$.

\begin{prop}\label{Proph}
	The holonomy algebras $\hol_x(\nabla^E)$ and $\hol_x(\nabla^{g,E})$ coincide.
\end{prop}
\begin{proof}
	Let $\bar \nabla$ be one of the connections 	$\nabla^E$ and $\nabla^{g,E}$. Recall that the covariant derivative of a field  $F$ of endomorphisms of $E$ is given by 
	$$\bar\nabla_XF=[\bar\nabla_X,F].$$
	Let $\mathcal{A}$ be the bundle defined  in the following way: 
	$$\Gamma(\mathcal{A}) = \{ A \in \so(E) \, | \, [A, \omega] = 0 \}.$$
	It is obvious that the both connections $\nabla^{g,E}$ and $\nabla^{E}$
	preserve this bundle, and, moreover, the restrictions of these connections to $\mathcal{A}$ coincide. Also it is obvious that the curvature tensors of the connections $\nabla^{g,E}$ and $\nabla^{E}$ on $E$ coincide and take their values in $\mathcal{A}$.  
	By the Ambrose-Singer theorem on holonomy, the holonomy algebra $\hol_x(\bar\nabla)$ is spanned by the operators of the form
	$$ \tau_\gamma ^{-1} (R^{\bar\nabla}_y(X, Y)), $$
	where $\gamma$ is a piecewise smooth curve starting at $x$ with an end-point $y$, 
	$\tau_\gamma$ is the parallel transport along  $\gamma$ with respect to the connection $\bar\nabla$ in the bundle $\mathcal{A}$, and $ X,Y \in T_yM.$ The statement of the proposition is now obvious. 
\end{proof}

Recall that the subalgebra $\hol_x(\nabla^{g,E})\subset\so(n)$ is the holonomy algebra of the Levi-Civita connection on a Riemannian manifold~\cite{L07}. Next, the subalgebra $\h=\hol_x(\nabla^E)=\hol_x(\nabla^{g,E})\subset\so(n)$ annihilates the 2-from $\omega_x\in\wedge^2 E_x$. Consider the orthogonal composition
$$E_x=\ker \omega_x\oplus (\ker \omega_x)^\bot.$$
Since $\h$ annihilates $\omega$, it preserves this decomposition.
According to the de~Rham decomposition theorem \cite{Besse},
$$\h=\h_1\oplus\h_2,$$
where  $\h_1\subset \so(\ker \omega_x)$ and $\h_2\subset \so((\ker \omega_x)^\bot)$ are Riemannian holonomy algebras. Moreover, since $\omega_x|_{(\ker \omega_x)^\bot}$ is non-degenerate, it holds that $\h_1\subset\mathfrak{u}((\ker \omega_x)^\bot)$.

If the holonomy algebra $\g\subset\so(1,n+1)_{\Real p}$ is not weakly irreducible, then it is obvious that there exists a $\g$-invariant orthogonal decomposition of the tangent space
$$ T_xM=\mathR^{1, n + 1} = \mathR^{1, k + 1} \oplus \mathR^{n - k} $$
and the corresponding decomposition
$$ \g = \mathfrak{a} \oplus \mathfrak{b}$$
such that $\mathfrak{a}\subset \so(1,k+1)_{\Real p}$  is a weakly irreducible holonomy algebra of the Levi-Civita connection of a Lorentzian manifold admitting a parallel isotropic vector field, and  $\mathfrak{b}\subset\so(n-k)$ is the holonomy algebra of the  Levi-Civita connection of a Riemannian manifold.
The subalgebra $(\pr_{\so(k)}\mathfrak{a}) \oplus \mathfrak{b}\subset\so(n)$ annihilates an element from $\wedge^2\Real^n$
corresponding to the parallel bivector $\omega$ on~$E$.
Using the results from \cite{GalRMS} it is easy to see that each $R\in\R^T(\g)=\R(\g)$ is given by
\begin{align*} R(q,X)&=P(X)+p\wedge K(X),\quad X\in\Real^k,\\
R(X,Y)&=R_1(X,Y)-p\wedge (P(X)Y-P(Y)X),\quad X,Y\in\Real^{k},\\
R(X,Y)&=R_2(X,Y),\quad X,Y\in\Real^{n-k},
\end{align*} where $R_1+R_2\in\R(\h)=\R(\pr_{\so(k)}\g)\oplus\R(\fb)$,
$K:\Real^k\to\Real^k$ is a symmetric linear map, and $P\in\mathcal{P}(\pr_{\so(k)}\g)$,
where elements of   $\mathcal{P}(\pr_{\so(k)}\g)$ are linear maps from
$\Real^k$ to $\pr_{\so(k)}\g$ satisfying
$$\underset{X Y Z}{\Su}g(P(X)Y,Z)=0.$$

The following two examples show that in general the holonomy algebras of the connections $\nabla$ and $\nabla^g$ do not coincide.

\begin{ex}{\rm 
Let  $v, x^1, \dots, x^n, u$ be coordinates on $M = \mathR^{n + 2}$. Suppose that $ n = k + 2m,$ $k \geq 0$ and let $g$ be the metric given by  
$$ g = 2 dv du + \sum_{i=1}^n (dx^i)^2 + H (du)^2, $$
where $H = \sum_{j=1}^k (x^j)^2.$ It is clear that $(M,g)$ is a product of an indecomposable Cahen-Wallach space of dimension $k+2$ and of the Euclidean space of dimension $n-k=2m$. The non-zero curvature operators
are
$$R^g(\partial_u,\partial_{x^{i}})=\partial_v\wedge \partial_{x^{i}},\quad i=1,\dots,k.$$
In particular, $\hol(\nabla^g) = p \wedge \mathR^k$.

Consider the 3-vector $$ T = - 2 \partial_v \wedge \sum_{i=1}^m 
\partial_{x^{k+2i}} \wedge \partial_{x^{k+2i+1}}$$  and the connection $\nabla=\nabla^g+\frac{1}{2}T$. It is not hard to check that $\nabla T=0$. For the connection $\nabla$ it holds that $\nabla\partial_v=0$ and
$$R(\partial_u,\partial_{x^{i}} ) = \partial_v \wedge \partial_{x^{i}}, \, j = 1, \dots, n.$$
This and Proposition \ref{Proph} imply that $ \hol(\nabla) = p \wedge \mathR^n. $
}
\end{ex}

\begin{ex}{\rm 
Let $(M, g)$ be as in Example 1 with $H = \sum_{i = 1}^n (x^i)^2.$ 
It is clear that $ \hol(\nabla^g) = p \wedge \mathR^n$.
Let now $T=\partial_v \wedge \sum_{i=1}^m 
\partial_{x^{k+2i}} \wedge \partial_{x^{k+2i+1}}$ and consider  the connection $\nabla$ as above. The non-zero curvature operators of $\nabla$ are 
$$R(\partial_u,\partial_{x^{i}})=\partial_v\wedge \partial_{x^{i}},\quad i=1,\dots,k.$$ Moreover, the bivectors $\partial_v\wedge \partial_{x^{i}},$ $i=1,\dots,k$, are $\nabla$-parallel. We conclude that $ \hol(\nabla) = p \wedge \mathR^k. $
}
\end{ex}

Next we note that the following two known examples are in the scope of the present section.

\begin{ex}{\rm It is known that (see, e.g., \cite{FF09})
a flat Lorentzian geometry $(M, g, \nabla)$ with closed skew-symmetric  torsion $T$ is locally isometric to a Lie group with a bi-invariant metric. In the case when $(M, g)$ is locally indecomposable, $(M, g)$ is locally isometric either to $\SO(2,1)$ with a multiple of its Killing form or to the simply connected Lie group with the following Lie algebra:
$$ \mathfrak{d}_{2n + 2} = \mathR^{2n} \oplus \mathR \oplus \mathR, $$
with the Lie bracket given by 
$$ [(v, v^-, v^+), (w, w^-, w^+)] = (v^- J(w) - w^- J(v), 0, v \cdot J(w)), $$
where $J$ is a non-degenerate skew-symmetric endomorphism of $\mathR^{2n}.$
The Lie group is diffeomorphic to $\Real^{2n+2}$,  and the metric is given by
$$ g = 2 dv du + \sum_{i=1}^{2n} (dx^i)^2 - (J x, J x)(du)^2;$$
the torsion is equal to 
$$ T = du \wedge \omega, $$ 
where $\omega$ is the 2-form associated  to $J$.
}
\end{ex}

\begin{ex}{\rm  
The metric of a \emph{pp-wave} locally takes the form
\begin{equation} 
g = 2 dv du + \sum_{i=1}^n (dx^i)^2 + H(du)^2, 
\end{equation} 
where $H=H(x^1,\dots,x^n,u)$. A
\emph{plane wave} is a pp-wave with the function $$H=\sum_{i,j=1}^n A_{ij}(u) x^i x^j,$$ 
where $A_{ij}(u)$ is a symmetric matrix.
Regular homogeneous plane wave metrics were classified in \cite{BOL}, see also \cite{FFSPPM}. In \cite{GL} it is shown that homogeneous plane waves are reductive. 
Each regular homogeneous plane wave metric on $\Real^{n+2}$ has the form
$$  g = 2 dv du + \sum_{i=1}^n (dx^i)^2 + A(e^{-uF}x, e^{-uF}x)(du)^2,  $$
where $A$ is a symmetric bilinear form and $F$ is a constant skew-symmetric matrix. The AS-structure is defined by the 3-from
$$ T = du \wedge \omega, \quad\omega = 2\sum_{1\leq i<j\leq n} F_{ij} dx^i \wedge dx^j. $$
The curvature tensor of the corresponding connection $\nabla$
is given by $$R(\partial_u,\partial_{x^i})=
\frac{1}{4}\partial_v\wedge\Big(2(e^{-uF})^\intercal A e^{-uF} (\partial_{x^i})-F^2(\partial_{x^i})\Big),$$
see Section~\ref{SecHom} below. Since $\nabla R=0$, the holonomy algebra $\g$  of the connection $\nabla$ is given by the image of $R$ at any point. Considering a point with $u=0$, we see that $\g=p\wedge \im(2A-F^2)$.
}
\end{ex}

\section{On the reducible case}\label{SecRed}

Let $(M, g, \nabla)$ be a Lorentzian geometry with parallel skew-symmetric torsion $T$.  Let $\g\subset\so(1,n+1)$ be the holonomy algebra of the connection $\nabla$ at a point $x\in M$.

The geometry $(M, g, \nabla)$ is called \emph{reducible} if the holonomy algebra   $\g\subset\so(1,n+1)$ of the connection $\nabla$ is \emph{not} weakly irreducible, i.e., $\g$ preserves a proper non-degenerate subspace of the tangent space. It is clear that in this case there exists a non-trivial $\g$-invariant orthogonal decomposition of the tangent space
\begin{equation}\label{dec0}T_xM=L\oplus E.\end{equation}
The geometry $(M, g, \nabla)$ is called \emph{decomposable}
if the holonomy algebra   $\g\subset\so(1,n+1)$ preserves an orthogonal decomposition \eqref{dec0} such that it holds 
$$T_x\in\wedge^3 L\oplus \wedge^3 E.$$ Otherwise we say that the geometry is \emph{indecomposable}.

Recall that according  to the de~Rham-Wu Theorem (see, e.g., \cite{Besse}),  a pseudo-Riemannian manifold $(N,h)$ with  non weakly irreducible holonomy algebra is locally a product of pseudo-Riemannian manifolds of lower dimension. As it is known (see, e.g., \cite{CMS21}), such decomposition generally does not exist for connections with torsion.
Indeed, in the case of the Levi-Civita connection,  non-degenerate holonomy-invariant subspaces of the tangent space determine  (local) parallel distributions, which are involutive, and the manifold  becomes a local  product of the  integral submanifolds of these distributions. In the presence of a parallel skew-symmetric torsion $T$, 
the $\nabla$-parallel distributions defined by the $\g$-invariant subspaces $L$ and $E$ are involutive if and only if $T_x(L,L)\subset L$ and $T_x(E,E)\subset E$. If these conditions are satisfied, then  $L$ and $E$ are preserved by the holonomy algebra of the Levi-Civita connection, and by the Wu Theorem,  $(M,g)$ is locally a product of a Lorentzian manifold $(M_1,g_1)$ and  of a Riemannian manifold $(M_2,g_2)$.
The connection $\nabla$ induces metric connections on $(M_1,g_1)$ and $(M_2,g_2)$ with parallel skew-symmetric torsions $T_1$ and $T_2$ such that $T=T_1+T_2$.  Thus a geometry $(M,g,\nabla)$ is decomposable if and only if  it may be locally decomposed  as a product of two other geometries.

Let $\g$ be the holonomy algebra of a geometry with parallel skew torsion, and let~$W$ be a $\g$-invariant subspace of the tangent space.
In terminology of \cite{CMS21} the subspace $W$ (and the corresponding parallel distribution) is called \emph{horizontal} if $\g\cap\so(W)\neq 0$, and it is called \emph{vertical} otherwise.  

Let us formulate one of the main results from \cite{CMS21} (see Theorem 4.8 and Remark 3.15 in \cite{CMS21}).
Suppose that $(M, g, T)$ is a pseudo-Riemannian geometry with parallel skew-symmetric torsion such that its holonomy algebra preserves the orthogonal decomposition of the tangent space
$$T_xM=F_1\oplus F_2$$ and such that the torsion $T$ satisfies
\begin{equation}\label{TV1V2} T\in \wedge^3 F_1\oplus\wedge^3 F_2\oplus (\wedge^2 F_2\,\wedge F_1).
\end{equation} 
Then the vector  space $F_1$ determines a (local) parallel and involutive distribution, i.e., we get a foliation $\mathfrak{F}$ on $M$
and the projection to the leaf space appears
$$M\to M/\mathfrak{F}.$$
The space $M/\mathfrak{F}$ is locally a smooth manifold with a canonically induced geometry with parallel skew-symmetric torsion, i.e.,
the metric $g$ and the component of the torsion from $\wedge^3 F_2$ are projectable to the base $M/\mathfrak{F}$. Moreover,  $M/\mathfrak{F}$  locally admits a geometry with parallel curvature in the sense of  \cite[Definition 4.7]{CMS21}, this may be expressed as a specific structure on a local principal bundle over $M/\mathfrak{F}$ satisfying certain technical conditions. Conversely, \cite[Theorem 5.1]{CMS21} provides a construction allowing to locally reconstruct  the initial geometry on $M$ from the local geometry with parallel curvature on $M/\mathfrak{F}$.

\section{A construction}\label{secconstr}

In this section we give a construction that will be used later in order to provide examples of geometries with special holonomy. At the end of the section we explain the relation of this construction to constructions from \cite{CMS21} and \cite{Storm3}.

Let  $\nabla^0$ be a metric connection with a parallel skew-symmetric torsion $T_0$ on a pseudo-Riemannian manifold  $(M_0,g_0)$ of signature $(r,s)$. Let $\fb_0\subset\so(r,s)$ be the holonomy algebra of the connection $\nabla^0$ at a point $x_0\in M_0$.
Let $\n\subset\so(r,s)$ be a commutative subalgebra commuting with $\fb_0$ and such that  $\sigma \cdot T_0 = 0$ for all $\sigma \in \n$. Denote by $k$ the dimension of $\n$. We fix an arbitrary basis $\sigma_1,\dots,\sigma_k$  of $\n$. Denote by $L_0$ a copy of the vector space $\n$. Let $h$ be an arbitrary pseudo-Euclidean metric on $L_0$ of signature $(r_0,s_0)$.
Let $V_1,\dots,V_k$ be an orthonormal basis of $L_0$ (i.e., $h(V_i,V_j)=\epsilon_i\delta_{ij},$ $\epsilon_i=\pm 1$).
Since  $\n$ commutes with $\fb_0$, the elements $\sigma_i$ determine $\nabla^0$-parallel fields of endomorphisms on $M_0$, which we denote by the same symbols. The vectors $V_1,\dots,V_k$ will be considered as the constant vector fields on the manifold $L_0$.

\begin{lem}\label{parallel2form} Let $(N,g,\nabla)$ be a pseudo-Riemannian manifold with a metric connection $\nabla$ with a parallel skew-symmetric torsion $T$. Let $\sigma$ be a $\nabla$-parallel 2-form on $N$. Then the form $\sigma$ is closed if and only if $\sigma \cdot T = 0. $ 
\end{lem}

{\bf Proof.} Using the formulas for the exterior and covariant derivatives of a 2-form, we get the equality
\begin{multline*}
d\sigma(X,Y,Z) = (\nabla_X\sigma)(Y,Z)-(\nabla_Y\sigma)(X,Z)+ (\nabla_Z\sigma)(X,Y)   \\
+ \sigma(T(X,Y),Z) - \sigma(T(X,Z), Y) + \sigma(T(Y,Z),X). 
\end{multline*}
Next,
$$ \sigma(T(X,Y),Z) = g(\sigma T(X,Y), Z) = - g(T(X, Y), \sigma Z ) = -T(X, Y, \sigma Z).$$
Since $\nabla \sigma=0$, we obtain
$$d\sigma=\sigma\cdot T.$$ This proves the lemma.
\qed	

In what follows we assume that the 2-forms $\sigma_i$ are exact and we fix  1-forms $\kappa_i$ such that $$d\kappa_i=\sigma_i$$ (the indices $i,j$ will take the values $1,\dots,k$).

Consider the product $$M=M_0\times L_0.$$ The vector fields on $M_0$ and $L_0$ will be considered as vector fields on $M$. The letters $ W,X, Y, Z$ will denote vector fields on $M_0.$ 
Let $\mathcal{L}$ be the distribution on $M$ linearly generated by the vector fields of the form $$ \overline X = X-\sum_i\epsilon_i \kappa_i(X)V_i.$$
Let $\mathcal{L}_0$ be the distribution on $M$ linearly generated by the vector fields $V_1, \dots, V_k.$ 	  
The following equality holds
\begin{equation}\label{commutatorEquality}
[ \overline{X}, \overline{Y} ] = \overline{[X, Y]} - \sum_i \epsilon_i \sigma_i(X, Y)V_i .
\end{equation}

Each field  of endomorphisms $A$ on $M_0$ defines a field  of endomorphisms  $\overline A$ of $\mathcal{L}$ by the equality $$ \overline{A}(\overline{X}) = \overline{A(X)} .$$
We will extend $\overline{A}$ to $TM$ by setting $\overline A (V_i) = 0. $

Let us define a metric $g$ on $M$ such that $ \mathcal{L} $ is orthogonal to $ \mathcal{L}_0 $ and 
$$ g(\overline X, \overline Y) = g_0(X, Y), \quad g(V_i, V_j) = h(V_i,V_j). $$

Next we define a 3-form $T$ on $M$ by the following conditions:
\begin{align*}
T(\overline X, \overline Y, \overline Z) &= T_0(X, Y, Z), \\
T(\overline X, \overline Y, V_j) &= \sigma_j(X, Y), \\
T(V_i, V_j, \cdot ) &= 0. 
\end{align*}

Finally we define a connection $\nabla$ on $M$ to be the unique $g$-metric connection with the skew-symmetric torsion $T:$
\begin{equation} \label{connectionNabla}
\nabla = \nabla^g + \frac{1}{2}T .
\end{equation}

\begin{lem}
	\label{lemmaDistributionIsParallel}
	The distribution $\mathcal{L}$ is parallel with respect to the connection $\nabla.$
\end{lem}

\begin{proof}
	First we will prove that $g( \nabla_{\overline X}\overline{Y}, V_i) = 0$ for all $  V_i$ and for all $X, Y \in \Gamma(TM_0).$
		Applying the Koszul formula and  (\ref{commutatorEquality}), we get:
	\begin{multline*}
	2g( \nabla^g_{\overline X}\overline{Y}, V_i)=\,\overline X(g(\overline Y,V_i))+\overline Y(g(\overline X,V_i))- V_i(g(\overline X,\overline Y)) \\
	\quad +\,g([\overline X, \overline Y],V_i)-g([\overline X,V_i],\overline Y)-g([\overline Y,V_i],\overline X) \\
	= g([\overline X, \overline Y],V_i) = -\sigma_i(X, Y).
	\end{multline*}
	
	Therefore, 
	$$2g( \nabla_{\overline X}\overline{Y}, V_i) = -\sigma_i(X, Y) + T(\overline X, \overline Y, V_i) =  0. $$
	
	Similarly it holds that $g(\nabla_{V_j} \overline Y, V_i) = 0. $
	This proves the lemma.
	
\end{proof}

\begin{lem}
	\label{lemmaConnectionFormulas}
	The connection $\nabla$ is given by the following equalities:
	\begin{align*}
	\nabla_{\overline X} \overline Y &= \overline{\nabla^0_X Y}, \\
	\nabla_{V_j} \overline X &= \overline{ \sigma_j (X)}, \\
	\nabla V_j &= 0.
	\end{align*}

\end{lem}

\begin{proof} The first equality follows from (\ref{connectionNabla}) and Lemma \ref{lemmaDistributionIsParallel}:
	
	\begin{multline*}
	2g(\nabla_{\overline X} \overline Y, \overline Z) = 2g_0(\nabla^{g_0}_X Y, Z) + T(\overline X, \overline Y, \overline Z) = 2g_0(\nabla^{g_0}_X Y + \frac{1}{2}T_0(X, Y), Z)\\
	= 2g_0(\nabla^0_X Y, Z) = 2g(\overline{\nabla^0_X Y}, \overline Z).
	\end{multline*}
	
	We prove the second equality in the same way:
	\begin{multline*}
	2g(\nabla_{V_j} \overline X, \overline Z) = 2g(\nabla^g_{V_j} \overline X, \overline Z) + T(V_j, \overline X, \overline Z)\\
	= V_j(g(\overline X,\overline{Z}))+\overline X(g(V_j,\overline{Z}))- \overline{Z}(g(V_j,\overline X))
	\\+ g([V_j, \overline X],\overline{Z})-g([V_j,\overline{Z}],\overline X)-g([\overline X,\overline{Z}],V_j) + \sigma_j(X, Z)\\
	=  -g([\overline X,\overline{Z}],V_j) + \sigma_j(X, Z) = 2\sigma_j(X, Z) = 2 g( \overline{ \sigma_j(X)}, \overline Z).
		\end{multline*}
	The equality $\nabla_{V_i} V_j = 0 $ follows easily from the Koszul formula.
	Finally, 
	\begin{multline*}
	2g(\nabla_{\overline{X}} V_j, \overline Z) = 2g(\nabla^g_{\overline{X}} V_j, \overline Z) + T(\overline{X}, V_j, \overline Z) \\
	= \overline{X}(g(V_j,\overline{Z}))+V_j(g(\overline{X},\overline{Z}))- \overline{Z}(g(\overline{X},V_j)) 
	\\+ g([\overline{X}, V_j],\overline{Z})-g([\overline{X},\overline{Z}],V_j)-g([V_j,\overline{Z}],\overline{X}) - \sigma_j(X, Z) \\
	= -g([\overline{X},\overline{Z}],V_j) - \sigma_j(X, Z) = 0.
	\end{multline*}
	This completes the proof of the lemma.
\end{proof}

\begin{lem}
	The torsion tensor $T$ of the connection $\nabla$ is parallel.
\end{lem}
\begin{proof}
	The lemma can be proved by direct computations like the following one: 
	
	\begin{multline*}
	(\nabla_{\overline W} T) (V_j, \overline Y, \overline Z) = {\overline W} ( T(V_j, \overline Y, \overline Z) ) \\- T(\nabla_{\overline W} V_j, \overline Y, \overline Z) - T(V_j, \nabla_{\overline W} \overline Y, \overline Z) - T(V_j, \overline Y, \nabla_{\overline W} \overline Z)\\
	= W (\sigma_j (Y, Z)) - \sigma_j(\nabla^0_W Y, Z) - \sigma_j(Y, \nabla^0_W Z) = (\nabla^0_W \sigma_j) (Y, Z) = 0.
	\end{multline*}
	
\end{proof}

\begin{lem}
	\label{curvatureTensor}
	The curvature tensor $R$ of $\nabla$ satisfies the following conditions:
	\begin{equation}
	\label{Rconstr}
	 R(\overline X, \overline{Y}) = \overline{R_0(X, Y)} + \sum_{i}\epsilon_i \sigma_i(X,Y)\overline{\sigma_i}, \quad R(V_i, \cdot) = 0. \end{equation}
\end{lem}
\begin{proof}
	The formulas above follow directly from \eqref{commutatorEquality} and  Lemma \ref{lemmaConnectionFormulas}.
\end{proof}

Now we find the holonomy algebra of the connection $\nabla$.

\begin{theorem} Let $\fb_0\subset\so(r,s)$ be the holonomy algebra of a metric connection $\nabla^0$ with a parallel skew-symmetric torsion $T_0$ on a simply connected manifold $M_0$. Let $\n\subset\so(r,s)$ be a commutative subalgebra commuting with $\fb_0$ and such that $\fb_0\cap \n=0$.  Suppose that the parallel 2-forms on $M_0$ defined by elements of $\n$ are exact. Let $\Real^{r_0,s_0}$ be a pseudo-Euclidean space of dimension $k$. Then the tangent space to $M$ at the point $x\in M$ may be identified with the pseudo-Euclidean space $$\Real^{r,s}\oplus \Real^{r_0,s_0};$$ the holonomy algebra of the connection $\nabla$ constructed just above annihilates the space $\Real^{r_0,s_0}$ and  coincides with
	$$\fb_0\oplus\n\subset\so(r,s)\subset\so(r+r_0,s+s_0).$$ 
\end{theorem}

\begin{proof}
	
	Let $\mathcal{A}$ be the following subbundle of $ \operatorname{End}(TM)$:  
	$$ \Gamma(\mathcal{A}) = \{ A \, | \, A \in \Gamma(\so(\mathcal{L})),  [A, \sigma_i] = 0,\,\, i=1,\dots,k \}. $$ We claim that the subbundle 
	$ \mathcal{A} $ is  parallel.
	Indeed, if $A$ is a section of $ \mathcal{A} $, then
	$$[\nabla_X A, \sigma_i] = \nabla_X [A, \sigma_i] - [A,\nabla_X \sigma_i] = 0. $$

	Let $x^1, \dots, x^n$ be local coordinates on $M_0$ and $y^1, \dots, y^k$ be coordinates on $L_0$ corresponding to the vector fields $V_1, \dots, V_k. $ Let $\partial_i = \frac{\partial }{\partial x^i}. $ Then $\overline{\partial_1}, \dots, \overline{\partial_n}, V_1, \dots, V_k$ is a local frame on $M$ such that $\mathcal{L}$ is locally linearly generated  by $ \overline{\partial_1}, \dots, \overline{\partial_n} $.
	The dual coframe of that frame is $ dx^1, \dots, dx^n, \omega^1, \dots, \omega^k$, where $\omega^j = dy^j + \kappa_j $.
	With respect to the chosen frame, the Christoffel symbols of $\nabla$ are the following:
	$$ \Gamma_{i j}^{l} = {\Gamma_0}_{ij}^{l}, \quad \Gamma_{n + \alpha\, j}^{l} = \sigma_{\alpha j}^l, $$ 
	and all other Christoffel symbols are zero.

	Let $\gamma(t)$, $0\leq t\leq 1$, be a smooth curve in $M,$ $x = \gamma(0), y = \gamma(1).$
	The tangent vector field to $\gamma(t)$ may be decomposed as
	$$ \dot{\gamma}(t) = f^j(t)\overline{\partial_j} + k^\alpha(t)V_\alpha $$ for some functions
	$f^j(t)$, $k^\alpha(t)$.
	Let $A(t)=A_{i}^j(t) dx^i \otimes \overline{\partial_j}$ be a section of $\mathcal{A}$ along the curve $\gamma(t).$ We will denote by $\dot{A}(t)$ the field $ \dot A_{i}^j(t) dx^i \otimes \overline{\partial_j}.$ Denote by ${\Gamma_0}_j$ the matrix with elements $ ( {\Gamma_0}_{ji}^{l} ) .$
	The parallel transport equation in the bundle  $\mathcal{A}$ takes the form
	$$
	\dot{A}(t) + f^j(t) [{\Gamma_0}_{j}, A(t)] + k^\alpha(t) [{\sigma}_{\alpha}, A(t)]=0,
	$$
	which is just
	\begin{equation}
	\dot{A}(t) + f^j(t) [{\Gamma_0}_{j}, A(t)] = 0. 
	\end{equation}
	This equation coincides with the parallel transport equation for the field of endomorphisms $B(t)=A_{i}^j(t) dx^i \otimes \partial_j$ on $M_0$ along the curve
	$\pi\circ\gamma$, where $\pi:M\to M_0$ is the projection.
	Denote by $ \tau_{\gamma}$ and $ \tau_{\pi\circ \gamma} $ the parallel transports of the connections in the bundles of endomorphisms under the consideration.
	We conclude that $$\tau_\gamma A_x = \overline{\tau_{\pi \circ\gamma} B_{\pi(x)}}, $$ 
	where $A_x\in\mathcal{A}_x$, $B_{\pi(x)}$ is the corresponding endomorphism  of $T_{\pi(x)}M_0$,  $\overline{\tau_{\pi \circ\gamma} B_{\pi(x)}}$ is the element of $\mathcal{A}_{y}$ corresponding to the endomorphism  $\tau_{\pi \circ\gamma} B_{\pi(x)}$ of $T_{\pi(y)}M_0$.

	By the Ambrose-Singer theorem on holonomy, the holonomy algebra $\mathfrak{g}$ of $M$ at $x$ is spanned by the operators of the form:
	$$ \tau_\gamma ^{-1} (R_y(U,V )), $$
	for all  possible piecewise smooth curves $\gamma$ starting at the point $x$ with an end-point $y$ and all $U,V \in T_yM.$ Note that here $\tau_\gamma$ is the parallel transport in the bundle $\mathcal{A}$. By the above considerations we get
	\begin{equation}
	\label{holonomyGenerator}
	\tau_\gamma ^{-1} (R_y(\overline X, \overline Y)) =  
	\quad \overline{(\tau_{\pi \circ\gamma} )^{-1} (R_0)_{\pi (y)}(X, Y)} + \sum_{j}\epsilon_j \sigma_j(X,Y)\overline{\sigma_j}
	\end{equation} for all $ X,  Y\in T_{\pi(y)}M_0$.
	
	Let us prove that $\g$ contains $\n$. Fix an endomorphism $\sigma_i\in\n$. Denote by $\n_i\subset\n$ the vector subspace spanned by the 
	endomorphisms $\sigma_1,\dots,\sigma_{i-1},\sigma_{i+1},\dots,\sigma_k$.
	We claim that there exists $\xi_i\in\so(r,s)=\so(\mathcal{L}_x)$ such that
	$\xi\in(\fb_0\oplus\n_i)^\bot$ and $g_x(\xi_i,\sigma_i)\neq 0$. Indeed, if this is not a case, it holds $(\fb_0\oplus\n_i)^\bot\subset \sigma_i^\bot$, which would imply $\sigma_i\in   \fb_0\oplus\n_i$ and give a contradiction.
	Consider such a vector $\xi_i$. Using Lemma~\ref{curvatureTensor}, we get
	$$R_x(\xi_i)=\epsilon_i g_x(\xi_i,\sigma_i)\sigma_i.$$ This implies that $\sigma_i\in\g$, and $\n\subset\g$. 
	Now it is obvious that \eqref{holonomyGenerator} implies  $\g=\fb_0\oplus\n$. This proves the theorem. 
	
\end{proof}

As we explained in Section \ref{SecRed}, \cite[Theorem 5.1]{CMS21} gives a construction of geometry with parallel skew torsion from a geometry with parallel curvature in the sense of \cite[Definition 4.7]{CMS21}. Let us consider the just constructed structure $(M,g,\nabla)$. Let $x\in M$.
We have the holonomy-invariant decomposition 
$$T_xM=\mathcal{L}_x\oplus\mathcal{L}_{0x},$$
and the torsion satisfies the condition
$$T_x\in\wedge^3\mathcal{L}_x\oplus (\wedge^2\mathcal{L}_x\otimes \mathcal{L}_{0x}).$$
The distribution $\mathcal{L}_0$ is involutive, and the corresponding foliation $\mathfrak{F}$ consists of the fibers of the projection $M_0\times N_0\to M_0$, i.e., $M/\mathfrak{F}=M_0$. The induced geometry with parallel skew torsion coincides with the initial $(M_0,g_0,\nabla^0)$. Note that $L_0$ may be consider as an Abelian Lie group with the Lie algebra $\n$. 
Now, the geometry with parallel curvature on $M_0$ is defined by the trivial $L_0$-principal bundle $$M_0\times L_0\to M_0$$ and the connection form $\gamma$, where
$$\gamma(V_i)=\sigma_i\in \n \quad \text{and}\quad \gamma|_\mathcal{L}=0.$$  Applying   
\cite[Theorem 5.1]{CMS21} to that structure on $M_0$, one gets exactly the structure $(M,g,\nabla)$. Thus our construction is a special case of   
\cite[Theorem 5.1]{CMS21}. At the same time we avoid the technical conditions on the principle bundle, and the explicit expression of the connection allows us to compute torsion, curvature and holonomy.  

Let us recall a construction from~\cite{Storm3} (with a slight modification for the pseudo-Riemannian case). Let $(\m,R_0,T_0)$ be an infinitesimal model of a naturally reductive homogeneous space.
Let
 $\n\subset\so(\m)$ be a subalgebra commuting with $\fb_0=\im R_0\subset\so(\m)$ and 
 annihilating $R_0$ and $T_0$. Denote by $L_0$ a copy of the vector space $\n$. Suppose that there exists an  $\n$-invariant pseudo-Euclidean metric $h$ on $L_0$. Let $V_1,\dots,V_k$ be an orthonormal basis of $L_0$, and let $\sigma_1,\dots,\sigma_k$ be the corresponding basis of $\n$. Denote by $\bar\sigma_1,\dots,\bar\sigma_k$ the corresponding elements for the representation of $\n$ in $\m\oplus L_0$.
Define the tensors 
$$T=T_0+\sum_{i=1}^k \sigma_i\wedge V_i+2T_\n,$$
$$R=R_0+\sum_{i=1}^k \epsilon_i\bar\sigma_i\circ\bar\sigma_i,$$
where the 3-form $T_\n$ corresponds to the Lie bracket on $\n$, and $\epsilon_i=h(V_i,V_i)$. Then the triple 
$(\m\oplus L_0,R,T)$ is an infinitesimal model. As in the proof of the theorem above, if $\fb_0\cap\n=0$, then $\im R=\fb_0\oplus\n$. 
If the Lie algebra $\n$ is commutative, then $T_\n=0$, and the expressions for the just given $T$ and $R$ coincide with the expressions for the torsion and the curvature from the construction of this section.

\section{Reducible case: $L$ is vertical and $\dim L\geq 4$} \label{Secvert}

Before we start to study the general reducible structures, we consider the case when the Lorentzian part is vertical in the sense of \cite{CMS21}, this will allow to better understand Theorem \ref{ThdimL>3} from the next section and also simplify its proof.

\begin{theorem}\label{Thvert}
Let $(M, g, \nabla)$ be a Lorentzian geometry with a parallel skew-symmetric torsion $T$,  $x \in M,$ and let $\frakg\subset\so(1,n+1)$ be  the holonomy algebra of the connection $\nabla$ at the point $x$. Suppose that the holonomy representation of $\frakg$ in $T_x M$ decomposes into a non-trivial orthogonal direct sum 
\begin{equation}
\label{tangentSpaceDecomposition}
T_x M = L \oplus E
\end{equation}
of $\frakg$-modules $L = \mathR^{1, k + 1}$, $k\geq 2$, and $E = \mathR^{n - k}$ such that the representation of $\g$ in the Lorentzian part $L$ is weakly irreducible and $\frakg \cap \so(L) = 0,$ i.e. the Lorentzian part is vertical in the sense of~\cite{CMS21}. Choose a Witt basis $p, e_1, \dots, e_k, q$ in $L.$ Then the following holds:
\begin{itemize}
	\item 
There exists an orthogonal decomposition
$$E = E_1 \oplus E_0,$$
 $\dim E_0 = k,$ an orthonormal basis $V_1,\dots, V_k$ of $E_0$,
 a subalgebra $\fb_0 \subset \so(E_1)$, which annihilates some multivectors
  $$\omega_{E_1}\in\wedge^3 E_1,\quad \theta_1,\dots,\theta_k\in\so(E_1),\quad\lambda\in\wedge^2E.$$
  Moreover, the endomorphisms $\theta_1,\dots,\theta_k$ are mutually commuting; the Lie algebra $\fb_0$ is Berger subalgebra of $\so(E)$   with the torsion $\omega_{E_1}$.
    It holds that $\theta_i\cdot \lambda=0$, $\theta_i\cdot\omega_{E_1}=0$ and  $\lambda\cdot\omega_E=0$, where
     $$\omega_E=\omega_{E_1} + \sum_{i = 1}^k \theta_i \wedge V_i \in \wedge^3E.$$
  
  \item
  
  The holonomy algebra  $\g$ is of the form
$$ 
\frakg = \{ p \wedge \psi(A) + A \, | \, A \in \fb \},
$$
where $$\fb=\fb_0+\left<\theta_1,\dots,\theta_k\right>\subset\so(E_1)$$ and   $$\psi : \fb \to \mathR^k=\left<e_1, \dots, e_k \right>$$ is a surjective linear map such that $ \psi|_{[\fb, \fb]} = 0.$ 
\item
The torsion of $\nabla$ at the point $x$ has the form
$$
T = p \wedge \omega_{\Real^k} + p \wedge \sum_{i = 1}^k X_i \wedge V_i + p \wedge \lambda +\omega_{E_1}+ \sum_{i = 1}^k \theta_i \wedge V_i, 
$$ where $X_1,\dots, X_k$ is a basis of $\Real^k$, and  $\omega_{\Real^k}\in\wedge^2\Real^k$.
\item
Each algebraic curvature tensor $R\in\R^T(\g)$ is defined by the equalities
\begin{align*} R(q,X)&=\sum_{i = 1}^kg(X,X_i)\big( p\wedge\psi(\theta_i)+\theta_i\big),\quad X\in\Real^k,\\
R(Y,Z)&= p\wedge \sum_{i = 1}^k\theta_i(Y,Z)X_i+C(Y,Z),\quad Y,Z\in E_1, 
\end{align*} where $$C(Y,Z)=C_0(Y,Z)+\sum_{i = 1}^k\theta_i(Y,Z)\theta_i,$$
$C_0\in\R^{\omega_{E_1}}(\fb_0),$ and it holds $$\psi(C(Y,Z))=\sum_{i = 1}^k \theta_i(Y,Z)X_i.$$
\end{itemize}

\end{theorem}
\begin{proof}
Let us denote by $\g_L$ and $\fb $ the projections of the holonomy algebra $ \g $ to $\so(L)$ and $\so(E)$, respectively. 
The assumption $\frakg \cap \so(L) = 0$ allows us to regard $\g \subset \so(L) \oplus \so(E)$ as a graph of a surjective Lie algebra homomorphism $ \bar \psi : \fb \to \g_L.$
Therefore, $\g_L$ is a compact Lie algebra. According to the classification of weakly irreducible  subalgebras of $ \so(1, n + 1)$, $\g$ has to be conjugate to $ p \wedge \mathR^k.$
  Therefore we have
$$ \g = \{ p \wedge \psi(A) + A \, | \, A \in \fb \}, $$
where $\psi : \fb \to \mathR^k$ is a linear map such that $ p \wedge \psi(A) = \bar \psi(A).$ It is clear that $\psi$ is surjective and $\psi|_{[\fb, \fb]} = 0$.

Decomposition (\ref{tangentSpaceDecomposition}) defines the following $\g$-invariant decomposition of space of 3-vectors on~$T_xM$:
$$ {\wedge}^3 T_x M = {\wedge}^3 L \oplus ({\wedge}^2 L \otimes E) \oplus ( L \otimes {\wedge}^2 E) \oplus {\wedge}^3 E. $$
We get the corresponding decomposition for $T$:
\begin{equation}
\label{torsionDecomposition}
T = \omega_L + \mu + \nu + \omega_E. 
\end{equation}
By assumption, $\g\cdot T=0$. Therefore, $\g$ annihilates each of the four 3-vectors from (\ref{torsionDecomposition}). Results of Section 3 show that $ \omega_L = p \wedge \omega_{\Real^k},$ where $\omega_{\Real^k} \in \wedge^2 \mathR^k.$

\begin{lem}\label{Lemmunu}
	It holds
	 $$ \nu = p \wedge \lambda,\quad  \mu = p \wedge \sum_{j = 1}^k e_j \wedge \mu^{j}, $$
	 where $\lambda \in \wedge^2 E,$  $\fb \cdot \lambda = 0$;
	 	$ \mu^{j} \in E$, $\fb\cdot \mu^{j} = 0, \, j = 1, \dots, k.$
\end{lem}

\begin{proof}
Let us consider the tensor $\mu$. It can be uniquely written in the following form:
\begin{equation} \label{torsionComponentDecomposition}
\mu = p \wedge \sum_{j = 1}^{k} e_j \wedge \mu^{pj} + p \wedge q \wedge \mu^{pq} + \sum_{j = 1}^{k} e_j \wedge q \wedge \mu^{jq} + \sum_{i, j = 1}^k e_i \wedge e_j \wedge \mu^{ij},
\end{equation}
where $\mu^{pj}, \mu^{jq}, \mu^{pq}, \mu^{ij} \in E$, $\mu^{ij} = - \mu^{ji}.$ Pick an element $ \xi = p \wedge \psi(A) + A $ of $\g .$ Consider the equation $$ \xi \cdot \mu = 0. $$ Using (\ref{torsionComponentDecomposition}) it can be shown that the equation above comes down to the following system:
\begin{align*}
 \sum_{j=1}^k e_j \wedge A\mu^{pj} + \psi(A) \wedge \mu^{pq} - 2\sum_{i,j = 1}^k (e_i, \psi(A))e_j \wedge \mu^{ij} =& 0, \\
 A \mu^{pq} - \sum_{j=1}^k (e_j, \psi(A)) \mu^{jq} = 0, \quad
  \sum_{j=1}^k e_j \wedge A \mu^{jq} =& 0, \\
 \sum_{i = 1}^k e_i \wedge \psi(A) \wedge \mu^{iq} + \sum_{i, j = 1}^k e_i \wedge e_j \wedge A \mu^{ij} =& 0.
\end{align*}
 From the third equation it follows that $ A \mu^{jq} = 0.$ 
The fourth equation when applied to arbitrary  vectors $e_l, e_m$ gives
$$ 
(\psi(A), e_l) \mu^{m q} - (\psi(A), e_m) \mu^{l q} + A \mu^{l m} = 0.
$$
Applying $A$, we see that $ A(A \mu^{l m}) = 0.$ Consequently, $$0 = (A(A \mu^{l m}), \mu^{l m}) = -(A \mu^{l m}, A \mu^{l m}),$$ which implies $A \mu^{l m} = 0.$
Using the fourth equation again we see that $\mu^{iq} = 0$ for all $i.$
The second equation now is the following: $ A \mu^{pq} = 0.$
Applying $A$ to the first equation we obtain
$$ \sum_{j=1}^k e_j \wedge A(A \mu^{pj}) = 0, $$
from which it follows that $A \mu^{pj} = 0.$ 
The first equation comes down to
$$ \psi(A) \wedge \mu^{pq} - 2\sum_{i,j = 1}^k (e_i, \psi(A))e_j \wedge \mu^{ij} = 0. $$
Applying the equality to $\psi(A)$, we obtain
$$ (\psi(A), \psi(A))\mu^{pq} - 2\sum_{i,j = 1}^k (e_i, \psi(A))(e_j, \psi(A)) \mu^{ij} = 0. $$
The second term is zero because $\mu^{ij} = -\mu^{ji}.$ Finally we have  $\mu^{pq} = 0, $ and $ \mu^{ij} = 0 $ follows easily.
We have shown that
$$ \mu = p \wedge \sum_{j = 1}^k e_j \wedge \mu^{j}, $$
for some $ \mu^{j} \in E$ such that $ \fb\cdot \mu^{j} = 0, \, j = 1, \dots, k.$ The structure of the tensor $\nu$ may be found in a similar way.
\end{proof}

Let us denote by $E_0$ the vector subspace of $E$ spanned by the vectors $\mu^1, \dots, \mu^k.$ The orthogonal complement of $E_0$ in $E$ will be denoted by $E_1.$ We have just seen that $\fb$ annihilates $E_0$, i.e., $\fb\subset\so(E_1)$.

Let us consider an algebraic curvature tensor $R\in\R^T(\g)$. The equality \eqref{sympropR} implies 
$$R(p,\cdot)=R(\cdot|_{E_0},\cdot)=R|_{\wedge^2 \Real^k}=R|_{L\times E}=0.$$
Writing down the Bianchi identity for the vectors $q,e_i,e_j$ we get
\begin{multline*}
R(q, e_i)e_j + R(e_j, q)e_i = \omega_{\Real^k}(\omega_{\Real^k}(e_i), e_j)p - \omega_{\Real^k}(\omega_{\Real^k}(e_j), e_i)p\\ -(\mu^j, \mu^i)p + (\mu^i, \mu^j)p = 0. 
 \end{multline*}
This shows that the endomorphism $K : \mathR^k \to \mathR^k$  defined by the equality
$$ R(q,e_i)|_L = p\wedge K(e_i)$$
is symmetric.
The Bianchi identity written for the vectors $Z \in E,q, e_i $ implies
$$
 R(q, e_i)Z = \omega_E(\mu^i, Z),
$$
i.e.,
$$R(q, e_i)|_E=\omega_E(\mu^i),$$
which implies that $\omega_E(\mu^i)\in\fb\subset\so(E_1)$. Consequently,
$$\omega_E \in \wedge^3 E_1 + \wedge^2 E_1 \otimes E_0.$$

Considering the vectors $Y, Z \in E$ and $q$ we get
$$
R(Y, Z)q = (\lambda \cdot \omega_E)(Y, Z) + \sum_{j=1}^k \omega_E(\mu^j, Y, Z)e_j.
$$
This shows that $\lambda \cdot \omega_E = 0,$ and $$R(Y, Z)|_L = \sum_{j=1}^k (\omega_E(\mu^j) Y, Z)p\wedge e_j.$$
Using \eqref{sympropR}, we get $$R(q,e_i)|_E=\omega_E(\mu^i),$$
i.e., $$K(e_i)=\psi(\omega_E(\mu^j)).$$

From the Ambrose-Singer theorem on holonomy it follows  that $\g_L=p\wedge\Real^k$ is  spanned by the endomorphisms $R(X,Y)|_L $ for all $X, Y \in T_x M,$
i.e., the vectors   $\sum_{j=1}^k(\omega_E(\mu^j) Y, Z)p\wedge e_j$ and $\psi(\omega_E(\mu^i))$ span $\Real^k$. We claim that the vectors $\mu^i, \, i = 1,\dots,k$ are linearly independent. Indeed, suppose that there exist numbers $c_1, \dots, c_k$ such that $\sum_{i=1}^k c_i \mu^i = 0$ and not all $c_i$ are zero. Let $ X=\sum_{i=1}^k c_i e_i.$ Then $X\neq 0$ and it is not hard to see that $X$ is orthogonal to the vectors $\sum_{j=1}^k(\omega_E(\mu^j) Y, Z)p\wedge e_j$. Moreover, $$R(q,X)=p\wedge K(X)+\omega_E\left(\sum_{i=1}^k c_i\mu_i\right)=p\wedge K(X).$$
Consequently $K(X)=0$. Since $K$ is symmetric, its image is orthogonal to $X$. We see that $X$ is orthogonal to the image of the map $\psi$, and we get a contradiction proving the claim. Similar arguments show that
the map $\omega_E|_{E_0}:E_0\to\fb$ is injective. We see that $\dim E_0=k$.   
Let $V_1,\dots V_k$ be an orthonormal basis of $E_0$. Using it, we may rewrite $T$ in the from 
$$
T = p \wedge \omega_{\Real^k} + p \wedge \sum_{i = 1}^k X_i \wedge V_i + p \wedge \lambda +\omega_E,\quad \omega_E=\omega_{E_1}+ \sum_{i = 1}^k \theta_i \wedge V_i, 
$$ where $X_1,\dots,X_k$ is a basis of $\Real^k$, $\theta_1,\dots,\theta_k$ are linearly independent endomorphisms.
The condition $\g\cdot T$ implies $\fb\cdot\theta_i=0$, $\fb\cdot\lambda=0$, $\fb\cdot \omega_{E_1}=0$.
From above we see that $$R(q,X)|_E=\sum_{i=1}^k(X,X_i)\theta_i,\quad X\in\Real^k.$$ This implies that $\theta_i\in\fb$. We conclude that $\theta_i$ are mutually commuting. Let 
 $C=\pr_{\fb}\circ R|_{\wedge^2_E}.$
Writing down the Bianchi identity for 3 vectors from $E_1$, we get 
$$ \underset{X Y Z}{\Su} C(X, Y)Z= \underset{X Y Z}{\Su} \omega_{E_1}(\omega_{E_1}(X, Y), Z) +\underset{X Y Z}{\Su}\sum_{i=1}^k \theta_i(X,Y)\theta_i(Z).$$
We conclude that
$$C(X,Y)=C_0(X,Y)+\sum_{i=1}^k\theta_i(X,Y)\theta_i(Z),$$
where  $C_0\in\R^{\omega_{E_1}}(\fb)$.
This implies that the algebra $\fb_0$ generated by the images of the tensors $C_0$ defined by all $R\in\R^T(\g)$ is a Berger algebra with the torsion $\omega_{E_1}$. It is clear that $\fb_0$ commutes with the endomorphisms $\theta_i$, and $$\fb=\fb_0+\left<\theta_1,\dots,\theta_k\right>.$$
The Bianchi identity written for other triples of vectors does not give any new conditions. The theorem is proved.
\end{proof}

\begin{ex}{\rm
Let us construct some of the spaces  described in the above theorem. Let $(N_0,b_0,T_0)$ be a Riemannian  geometry with a parallel skew torsion $T_0$ and holonomy algebra $\fb_0\subset\so(m)$, $m=\dim N_0$. Let $(M_0,g_0,T_0)$ be the product of $(N_0,b_0,T_0)$ with the flat Minkowski space $\Real^{1,k+1}$. It is clear that the holonomy algebra of $(M_0,g_0,T_0)$ is $\fb_0\subset\so(m)\subset\so(1,k+m+1)$.
Let $\theta_1,\dots,\theta_k\in\so(m)$ be linearly independent mutually commuting endomorphisms commuting with $\fb_0$. Let $\n_0=\left<\theta_1,\dots,\theta_k\right>$. Suppose that $\n_0\cap \fb_0=0$. Consider the Lie algebra $\fb=\fb_0\oplus\n_0$. Let $$\psi:\fb\to\Real^k$$ be a surjective linear map which is zero on $\fb_0$.
Let finally $$\sigma_i=p\wedge\psi(\theta_i)+\theta_i,\quad i=1,\dots,k.$$
The construction of Section \ref{secconstr} gives us a geometry $(M,g,\nabla)$ with the torsion
$$T=p\wedge\sum_{i=1}^k\psi(\theta_i)\wedge V_i+ \sum_{i=1}^k\theta_i\wedge V_i+T_0$$
and the holonomy algebra
$$ 
\frakg = \{ p \wedge \psi(A) + A \, | \, A \in \fb \}\subset\so(1,m+2k+1).
$$
}
\end{ex}

Let us show that \cite[Theorem 4.8]{CMS21} may be applied to connections from this section. We use the notation of Theorem~\ref{Thvert}. Let $V_0\subset E$ be the subspace consisting of vectors annihilated by $\fb$. Let $V_1\subset E$ be the orthogonal complement to $V_0$ in $E$. We obtain the decomposition 
$$E=V_1\oplus V_0.$$ It is clear that
$$\lambda=\lambda_1+\lambda_0,\quad \lambda_1\in\wedge^2 V_1,\quad \lambda_0\in\wedge^2 V_0.$$ It holds that $E_0\subset V_0$.
Consider the holonomy-invariant decomposition
\begin{equation}\label{HVLorv} T_xM=F_1\oplus F_2=(L\oplus V_0)\oplus V_1.\end{equation}
The equality \eqref{TV1V2} holds true. Consequently $F_1$ determines a foliation $\mathfrak{F}$; the metric $g$  and the part of the torsion determined by $T|_{\wedge^3 V_1}$ define a geometry with parallel skew-torsion locally on the leaf space $M/\mathfrak{F}$.

\section{Reducible case: $\dim L\geq 4$} \label{SecdimL>3}

\begin{theorem}\label{ThdimL>3}
	Let $(M, g, \nabla)$ be a Lorentzian geometry with a parallel skew-symmetric torsion $T$,  $x \in M,$ and let $\frakg\subset\so(1,n+1)$ be  the holonomy algebra of the connection $\nabla$ at the point $x$. Suppose that the holonomy representation of $\frakg$ in $T_x M$ decomposes into a non-trivial orthogonal direct sum 
	\begin{equation}
	\label{tangentSpaceDecomposition3}
	T_x M = L \oplus E
	\end{equation}
	of $\frakg$-modules $L = \mathR^{1, k + 1}$, $k\geq 2$, and $E = \mathR^{n - k},$ such that the induced representation of $\g$ in the  Lorentzian part $L$ is weakly irreducible. Suppose that the torsion $T$ is not an element of $\wedge^3L\oplus\wedge^3E$. Denote by $\fb$ the projection of $\g$ to $\so(E)$. Choose a Witt basis $p, e_1, \dots, e_k, q$ in $L$ and set $\Real^k=\left<e_1,\dots,e_k\right>$. Then the following holds:
	\begin{itemize} 
			\item
		The holonomy algebra  $\g$ is of the form
		$$ 
		\frakg = \{ p \wedge \psi(A+B) + A+B \, | A\in\h,\, B \in \fb \}\zr p\wedge \Real^{k_1}$$
		$$=
		\left\{\left.
		\begin{pmatrix}
		0 & -X^\intercal	&	-\psi(A+B)^\intercal	&	0	&	0 \\
		0 & A	& 0 & X & 0 \\
		0 & 0	& 0 & \psi(A + B) & 0 \\
		0 & 0	& 0 & 0 & 0 \\
		0 & 0	& 0 & 0 & B 
		\end{pmatrix}
		\right| \begin{matrix}
		A\in\h,\\ B \in \fb,\\X\in\Real^{k_1}\end{matrix}
		\right\},
		$$
		where an orthogonal decomposition $$\Real^k=\Real^{k_1}\oplus\Real^{k_2}$$
		is fixed, $k=k_1+k_2$, $0\leq k_2\leq k$, 
		$\h\subset\so(k_1)$ is the holonomy algebra of the Levi-Civita connection of a Riemannian manifold,
		and $$\psi : \h\oplus\fb \to \mathR^{k_2}$$ is a surjective linear map such that $\psi|_{[\h, \h]}= \psi|_{[\fb, \fb]} = 0.$
				\item To decomposition \eqref{tangentSpaceDecomposition3} corresponds the decomposition of the torsion $T$ of the connection $\nabla$ at the point $x\in M$,
		$$T=p\wedge \zeta+\omega_E,$$ where $\zeta\in\wedge^2(\Real^k\oplus E)$ and $\omega_E\in\wedge^3E$.
				There exists a unique orthogonal decomposition
		$$E = E_1 \oplus E_0$$
		such that $\fb\subset\so(E_1)$,
		$$\omega_E=\omega_{E_1}+\varphi,\quad \omega_{E_1}\in\wedge^3 E_1,\quad \varphi\in (\wedge^2 E_1)\wedge E_0,$$
		and $$E_0=\varphi(\wedge^2 E_1)\cap \pr_E\zeta(\Real^k).$$
		There exists a   Berger subalgebra $\fb_0\subset\so(E_1)$   with the torsion $\omega_{E_1}$. The image $\varphi(E_0)\subset
		\wedge^2 E_1=\so(E_1)$ is a commutative subalgebra commuting with $\fb_0$, and it holds 
		$$\fb=\fb_0+\varphi(E_0).$$ Moreover,
		$$\fb\cdot \omega_{E_1}=0,\quad \fb\cdot\varphi=0,\quad \fb\cdot\zeta=0,$$ $$\h\cdot\zeta=0,\quad \omega_{E_1}(\pr_{E_1}\zeta(\Real^k))=0,\quad (\pr_{\wedge^2 E}\zeta)\cdot \omega_{E}=0.$$
		 
			\item Let $V_1,\dots,V_l$ be an orthogonal basis of $E_0$. Let $\theta_i=\varphi(V_i)$, and $X_i=\pr_{\Real^k}\zeta(V_i)$.
		Then each algebraic curvature tensor $R\in\R^T(\g)$ is defined by the equalities
		\begin{align*} R(q,X)&=P(X)+p\wedge K(X)+\sum_{i = 1}^l g(X,X_i)\theta_i,\quad X\in\Real^k,\\
		R(X,Y)&=R_0(X,Y)-p\wedge (P(X)Y-P(Y)X),\quad X,Y\in\Real^{k_1},\\
		R(Y,Z)&= p\wedge \sum_{i = 1}^l\theta_i(Y,Z)X_i+C(Y,Z),\quad Y,Z\in E_1, 
		\end{align*} where $R_0\in\R(\ker\psi|_\h)$, $P\in\mathcal{P}(\h)$, $K:\Real^k\to\Real^k$ is a symmetric linear map such that
		$$\pr_{\Real^{k_2}}K(X)=\psi\left(P(X)+\sum_{i = 1}^l g(X,X_i)\theta_i\right);$$
	 next, $$C(Y,Z)=C_0(Y,Z)+\sum_{i = 1}^l\theta_i(Y,Z)\theta_i,$$
	$C_0\in\R^{\omega_{E_1}}(\fb_0),$ and it holds  $$\psi(C(Y,Z))=\sum_{i = 1}^l \theta_i(Y,Z)\pr_{\Real^{k_2}}X_i.$$
	\end{itemize}
	
\end{theorem}

\begin{proof}
Consider the $\g$-invariant decomposition
$$ \mathR^{1, n + 1} =L\oplus E= \mathR^{1, k + 1}\oplus\Real^{n-k}$$
as above. We assume that the induced representation of $\g$ in $\Real^{1,k+1}$ is weakly irreducible and $k\geq 2$. Using  arguments as in Lemmas \ref{lemwirr} and  \ref{Lemmunu}
it is not hard to show that the condition $T\not\in \wedge^3L\oplus\wedge^3E$
 implies that $\pr_{\mathfrak{so}(1, k + 1)}\g$ annihilates an isotropic vector $p$.  We see that $$\g\subset(\so(k)\zr p\wedge\Real^k)\oplus\so(E),\quad\pr_{p\wedge \Real^k}\g=p\wedge \Real^k.$$
  Let $V_0\subset E$ be the subspace consisting of vectors annihilated by $\g$. We obtain an orthogonal decomposition
 $$E = \mathR^m \oplus V_0 $$
 such that $\g$ does not annihilate any non-zero vector in $\Real^m$.
As in Section \ref{Secvert}, it is easy to see that the torsion has the form 
$$
T = p \wedge \omega_{\Real^k} + p \wedge \sum_{i = 1}^k e_i \wedge \mu_i + p \wedge \lambda + \omega_E ,
$$
where $ \omega_{\Real^k} \in \wedge^2 \mathR^k$, $\mu_i \in \mathR^m \oplus V_0$, $\omega_E\in\wedge^3E $ are such that
$ \omega_{\Real^k} $, $\sum_{i = 1}^k e_i \wedge\mu_i$, $\omega_E$
are annihilated by $\pr_{\so(n)}\g$.

Let us denote 
$$ \frakk_1 = {\so(k)} \cap \frakk,	\quad \frakk_2 = {\so(m)} \cap \frakk,$$
$$ \frakh = \pr_{\so(k)}\frakk,	\quad \fb = \pr_{\so(m)}\frakk.$$
Its easy to see that $\frakk_1 $ is an ideal in $\frakh$. Thus, there is a complementary ideal $\fraka$. The same holds for $\frakk_2 \subset \fb$ and we will denote the corresponding complementary ideal by $\frakc $.
The algebra $\frakk$ can be decomposed as a sum of ideals in the following way
$$ \frakk = \frakk_1 \oplus \frakk_2 \oplus \tilde\frakk, $$
and its easy to see that $ \tilde\frakk \cap \frakh = \tilde\frakk \cap \fb = 0$. This means that $\tilde\frakk$ is a graph of a Lie algebra isomorphism~$\fraka \to \frakc$.

From the invariance of $T$ we see that $\omega_E$ has the form 
$$ \omega_E = \omega_{\mathR^m} + \sum_{a=1}^{n_0} \theta_a \wedge V_a + \omega_0, $$
where $\omega_{\mathR^m}\in \wedge^3 \mathR^m$, $\theta_a \in \wedge^2 \mathR^m$, $\omega_0 \in \wedge^3 V_0$ and $V_1, \dots, V_{n_0}$ is an orthonormal basis of $V_0$.

\begin{lem}
	The Lie algebra $\tilde\frakk$ is trivial, i.e., $\frakk = \frakk_1 \oplus \frakk_2$.
\end{lem}
\begin{proof}
	
	Using the Bianchi identity as in Section \ref{Secvert} it is easy to see that the projections of curvature operators to $\so(n)$ are the following:
	\begin{align}
	&\pr_{\so(n)} R(e_i, q) =- P(e_i) -\sum_{a=1}^{n_0}(V_a, \mu_i) \theta_a \label{curvatureOperatorProjection1} \\
	&\pr_{\so(n)} R(e_i, e_j) = R_0(e_i, e_j) \in \frakk_1 \\
	&\pr_{\so(n)} R(X, Y) = C(X, Y) \in \frakk_2,\quad X,Y\in\Real^m,
	\end{align}
	and other curvature operators have zero projections to $\so(n)$;
	here $P\in\mathcal{P}(\h)$, $R_0\in\R(\h)$, $C\in\R^{\omega_E}(\fb).$
	
	We claim that the Lie algebra $ \tilde \frakk$ is commutative.
	The Lie algebra $ \tilde \frakk$ is spanned by the elements of the form $\pr_{\tilde \frakk} R(e_i,q)$, where $e_i \in \mathR^k$, $R \in \mathcal{R}^T(\g)$.
	We see from \eqref{curvatureOperatorProjection1} that $\sum_{a=1}^{n_0}(V_a, \mu_i) \theta_a \in \fb$. Since this element annihilates $T$, we obtain that $\sum_{a=1}^{n_0}(V_a, \mu_i) [\theta_a, \theta_b] = 0$. This proves the claim.

It is clear that the subalgebra $\h\subset\so(k)$ is spanned by the images of the elements $P\in\mathcal{P}(\h)$, $R_0\in\R(\h)$ defined by all $R\in\R^T(\g)$. This implies \cite{L07} that $\h\subset\so(k)$ is the holonomy algebra of the Levi-Civita connection of  a Riemannian manifold. The de~Rham theorem implies the decompositions  
		$$ \mathR^n = \mathR^{d_1}\oplus\dots\oplus \mathR^{d_l}\oplus\Real^{d_0},$$
	$$ \frakh = \h_1\oplus\dots\oplus\h_l,$$
	where $\frakh_\alpha\subset\so(d_\alpha)$ is irreducible, $\alpha=1,\dots, l$. We assume that basis $e_1, \dots, e_k$ is compatible with this decomposition.

	Suppose that $\frakh_\alpha' \ne 0$. Then the induced representation of $\frakh_\alpha'$ on $\Real^{d_\alpha}$ is either irreducible, or $\frakh'_\alpha=\so(m_\alpha)\oplus\so(2)$ and $\Real^{d_\alpha}=\Real^{m_\alpha}\otimes\Real^2$, $d_\alpha=2m_\alpha$ \cite{Besse}. Since $\h$ and $\tilde\frakk$ are ideals in $\frakk$, $\h'_\alpha$ is a semisimple ideal in $\h$, and $\tilde \frakk$ is commutative, we see that $\h'_\alpha\subset\frakk_1$.
	Since $T$ is invariant, we conclude that $\mu_i=0$ if $e_i\in\Real^{d_\alpha}$. From this and \eqref{curvatureOperatorProjection1} it follows that $\pr_{\so(n)} R(e_i, q) = -P(e_i)$ if $e_i\in\Real^{d_\alpha}$. This implies that $\h_\alpha\subset\frakk_1$.
	
	Suppose that $\frakh_\alpha^\prime = 0$. Then $\frakh_\alpha = \so(2)$, and $\mathR^{d_\alpha} = \mathR^2$. 
	Without loss of generality we may assume that $\alpha=1$. 
	For each $P\in\mathcal{P}(\so(2))$, the basis $e_1,e_2\in\Real^2$ may be chosen in such a way that 
	 $$ P(e_1) = -c  e_1 \wedge e_2,\quad  P(e_2) = 0,\quad c\in\Real.$$
	If $c = 0$, then  from \eqref{curvatureOperatorProjection1} it follows that $\pr_{\tilde\frakk} R(e_j, q) = 0$. Suppose that  $c \ne 0$. 
	We have 
	\begin{align}
	&\pr_{\so(n)} R(e_1, q) = -c \, e_1 \wedge e_2 - \sum_{a=1}^{n_0}(V_a, \mu_1) \theta_a, \label{curvatureOperatorProjection1_1} \\
	&\pr_{\so(n)} R(e_2, q) =  -\sum_{a=1}^{n_0} (V_a, \mu_2) \theta_a. 
	\end{align}
	From the equalities $ R(e_1, q) \cdot T=R(e_2, q) \cdot T = 0$ it follows that 
	\begin{align*}
	 c  \mu_1 &= \sum_{a=1}^{n_0}(V_a, \mu_1) \theta_a \mu_2 , \\
	 -c \mu_2 &= \sum_{a=1}^{n_0}(V_a, \mu_2) \theta_a \mu_1.
		\end{align*}
	Since $\sum_{a=1}^{n_0}(V_a, \mu_1)\theta_a,\sum_{a=1}^{n_0}(V_a, \mu_2)\theta_a\in\fb\subset\so(m)$, the  equations imply $\mu_1,\mu_2\in\Real^m$. Consequently, $(V_a,\mu_1)=(V_a,\mu_2)=0$, which implies $\mu_1=\mu_2=0$.
Now from \eqref{curvatureOperatorProjection1_1} we see that $\pr_{\tilde\frakk} R(e_1, q)=\pr_{\tilde\frakk} R(e_2, q) = 0$. This proves the lemma.
	\end{proof}

Thus, $$\pr_{\so(n)}\g=\h\oplus\fb,\quad \h\subset\so(k),\quad\fb\subset\so(m).$$
We already know  that $\pr_{\mathfrak{so}(1, k + 1)}\g\subset\so(1, k + 1)$ is weakly irreducible and it annihilates the isotropic vector $p$. So,  $\pr_{\mathfrak{so}(1, k + 1)}\g$ is a weakly irreducible subalgebra of type 2 or 4. We unify these algebras assuming $\psi=0$ for algebras of type 2. Now it is clear that $\g\subset\so(1,n+1)$ is as in the statement of the theorem.

The condition $\g\cdot T$ implies now that 
$$\sum_{i=1}^ke_i\wedge \mu_i=\sum_{e_i\in\Real^{d_0}} e_i\wedge \mu_i,\quad \mu_i\in V_0.$$
This expression may be considered as a map from $\Real^{d_0}$ to $V_0$.
Let us denote by $E'_0$ its image, $E'_0\subset V_0$. And consider an orthogonal decomposition $$E=E'_1\oplus E'_0.$$
Using this decomposition, $T$ may be rewritten as 
$$ p \wedge \omega_{\Real^k} + p \wedge \sum_{a = 1}^{n_0} X_a \wedge V_a +p \wedge \lambda +\omega_{E}, \quad \omega_E=\omega_{E'_1}+ \sum_{a = 1}^{n_0} \theta_a \wedge V_a+\omega_0,$$
where now $V_1,\dots, V_{n_0}$ is an orthonormal basis of $E'_0$, $X_1,\dots,X_{n_0}\in\Real ^k$ are vectors annihilated by $\h$,
$\lambda\in\wedge^2E$, $\omega_{E'_1}\in\wedge^3E'_1$, $\theta_1,\dots,\theta_{n_0}\in\so(E'_1)$, $\omega_0\in\wedge^3E'_0$.
By the construction, the vectors $X_1,\dots,X_{n_0}$ are linearly independent. As above, the Bianchi identity implies
$$\pr_{\so(m)}R(X,q)=-\sum_{a=1}^{n_0}(X,X_a)\theta_a,\quad X\in\Real^k.$$
Since the vectors $X_1,\dots,X_{n_0}\in\Real ^k$ are linearly independent, we conclude that $\theta_a\in\fb$ for all $a=1,\dots,n_0$.
The condition $\fb\cdot T=0$ implies that the endomorphisms $\theta_a$, $a=1,\dots,n_0$, are mutually commuting. As in Section \ref{Secvert} it can be shown that this implies $\omega_0=0$.

Consider now the expression  $\sum_{a = 1}^{n_0} \theta_a \wedge V_a$ as a map from $E'_0$ to $\fb$. Let $E_0\subset E'_0$ be the orthogonal complement to the kernel of this map. Let finally $E_1$ be the orthogonal complement to $E_0$ in $E$ and consider the decomposition $$E=E_1\oplus E_0.$$ Now it is obvious that $T$ may be represented as in the statement of the theorem. The expressions for the curvature tensor may be obtained from the Bianchi identity as we did it in Section \ref{Secvert}.

\end{proof}

	\begin{rem}
The torsion given in the statement of Theorem \ref{ThdimL>3} may be written in the form
$$
T = p \wedge \left(\omega_{\Real^k} +  \sum_{i = 1}^l X_i \wedge V_i+ \sum_{i = 1}^s Y_i \wedge U_i + \lambda\right)+\omega_E,\quad \omega_E=\omega_{E_1}+ \sum_{i = 1}^l \theta_i \wedge V_i, 
$$

where $\omega_{\Real^k}\in\wedge^2\Real^k$, $\h\cdot\omega_{\Real^k}=0$; $X_1,\dots, X_l,Y_1,\dots,Y_s\in\Real^k$ are linearly independent vectors annihilated by $\h$; $V_1,\dots, V_l$ is an orthonormal basis  of $E_0$; $\theta_1,\dots,\theta_l\in\so(E_1)$ are linearly independent mutually commuting elements commuting with $\fb_0$;  $U_1,\dots,U_s\in E_1$ are vectors annihilated by 
$$\fb=\fb_0+\left<\theta_1,\dots,\theta_l\right>\subset\so(E_1);$$
$\omega_{E_1}\in\wedge^3 E_1$, $\fb\cdot\omega_{E_1}=0$; and
$\lambda\in\wedge^2 E$, $\fb\cdot\lambda=0$. Moreover, it holds
 $\omega_{E_1}(U_i)=0$,  and $\lambda\cdot\omega_E=0$.

\end{rem}

\begin{rem} {\rm
Note that the decomposition \eqref{tangentSpaceDecomposition3} is not defined uniquely. Suppose that a decomposition \eqref{tangentSpaceDecomposition3} is fixed. Let as above $\fb\subset\so(E)$ be the projection of the holonomy algebra $\g$ to
$\so(E)$. Let $V_0\subset E$ be the subspace consisting of vectors annihilated by $\fb$. Let $V_1\subset E$ be the orthogonal complement to $V_0$ in $E$. We obtain the decomposition 
$$E=V_1\oplus V_0.$$ In notation of Theorem \ref{ThdimL>3} we have $V_1\subset E_1$, and $E_0\subset V_0$. Fix a vector $X_0\in V_0$ and consider the vector space $$\tilde E=V_1\oplus \tilde V_0,\quad \tilde V_0=\{X-(X,X_0)p|X\in V_0\}.$$
Let $\tilde L$ be the orthogonal complement to $\tilde E$ in $\Real^{1,n+1}.$ The vector space $\tilde L$ is spanned by the vectors $p,e_1,\dots,e_k,q+X_0 - \frac{1}{2}g(X_0,X_0)p.$
We obtain the new $\g$-invariant decomposition
\begin{equation}\label{decomnew}
\Real^{1,n+1}=\tilde L\oplus\tilde E.\end{equation}
Using the decomposition $E=V_1\oplus V_0$, the tensor $\omega_E$
may be written as $$\omega_E=\omega_{V_1}+\omega_{V_0}+\varphi,$$
where $\omega_{V_1}\in\wedge^3 V_1$, $\omega_{V_0}\in\wedge^3 V_0$ $\varphi\in\wedge^2V_1\,\otimes V_0.$
Let $Z_1,\dots, Z_s$ be an orthonormal basis of $V_0$. The tensor $\varphi$ may be written in the form 
$\varphi=\sum_{i=1}^s\eta_i\wedge Z_i$.
The condition $[\fb,\lambda]=0$ implies $$\lambda=\lambda_1+\lambda_0,\quad \lambda_1\in\wedge^2 V_1,\quad \lambda_0\in\wedge^2 V_0.$$
}\end{rem}

\begin{ex}{\rm
 Let $(M_1,g_1)$ be a Lorentzian manifold 
of dimension $k_1+k_2+2$, $k_1\geq 0$, $k_2\geq 0$ with the holonomy algebra $$\g_1=\{p\wedge \psi_1(A)+A|A\in\h\}\zr p\wedge \Real^{k_1}\subset\so(1,k_1+k_2+1),$$
where $\h\subset\so(k_1)$ is the holonomy algebra of a Riemannian manifold and $\psi_1:\h\to\Real^{k_2}$ is  an arbitrary linear map with $\psi_1|_{[\h,\h]}=0$. Such spaces exist according to \cite{GalRMS}.
Let $(N_0,b_0,T_0)$ be the product of $(M_1,g_1)$ and of a Riemannian  geometry with a parallel skew torsion $T_0$ and holonomy algebra $\fb_0\subset\so(m)$, $m=\dim N_0$.  
 It is clear that the holonomy algebra of $(M_0,g_0,T_0)$ is $\g_1\oplus\fb_0\subset\so(1,k_1+k_2+1)\oplus\so(m)$.
Let $\theta_1,\dots,\theta_{l}\in\so(m)$ be linearly independent mutually commuting endomorphisms commuting with $\fb_0$. Let $\n_0=\left<\theta_1,\dots,\theta_{l}\right>$. Suppose that $\n_0\cap \fb_0=0$. Let $\fb=\fb_0\oplus\n_0$. Let $$\psi_2:\fb\to\Real^{k_2}$$ be a  linear map which is zero on $\fb_0$ and such that the map $\psi=\psi_1+\psi_2:\h\oplus\fb\to\Real^{k_2}$ is surjective.
Let finally $$\sigma_i=p\wedge\psi(\theta_i)+\theta_i,\quad i=1,\dots,l.$$
The construction of Section \ref{secconstr} gives us a geometry $(M,g,\nabla)$ with the torsion
$$T=p\wedge\sum_{i=1}^l\psi(\theta_i)\wedge V_i+ \sum_{i=1}^l\theta_i\wedge V_i+T_0$$
and the holonomy algebra
$$ 
\frakg = \{ p \wedge \psi(A) + A \, |  A \in \h\oplus\fb \}\zr p\wedge \Real^{k_1}$$ contained in $\so(1,k_1+k_2+1)\oplus\so(m)\subset\so(1,k_1+k_2+1)
\oplus\so(l+m).
$
}
\end{ex}

For a geometry from the statement of the theorem,  one may consider the decomposition \eqref{HVLorv} and as in Section \ref{Secvert} apply to it  \cite[Theorem 4.8]{CMS21}.
Let us now show that some results from \cite{CMS21} do not hold true in the Lorentzian signature.

\begin{ex}{\rm
	Let $(M_0,g_0,\omega_E)$ be a Riemannian geometry with a parallel skew-symmetric torsion and a parallel 2-from $\lambda$ satisfying $\lambda\cdot\omega_E=0$. Denote by $\fb$ its holonomy algebra.  Let $(N_0,h_0,p\wedge\omega)$ 
	be a Lorentzian geometry as in Section \ref{Secpomega} and with the holonomy algebra $\g_L$. By Lemma \ref{parallel2form}, the 2-form $\lambda$ is closed. Let us suppose that $\lambda$ is exact, i.e., there exists a 1-from $\kappa$ such that $d\kappa=\lambda$.  Consider the manifold $M=M_0\times N_0$ and the product of the above geometries.
	Similarly, as in Section \ref{secconstr},
	consider the distribution $\mathcal{L}$ on $M$ spanned by the vector fields
	$$\bar X=X-\kappa(X)p,\quad X\in\Gamma(TM_0).$$
	As in Section \ref{secconstr} this defines a geometry with the torsion 	$$\bar \omega_E+p\wedge \omega+p\wedge \bar\lambda,$$
	where $\bar \omega_E$ and $\bar\lambda$ are the tensors on $\mathcal{L}$ corresponding to $\omega_E$ and $\lambda$.
	It is easy to check that the distributions $\mathcal{L}$ and $TN_0$ are parallel with respect to this connection, and the holonomy algebra of this connection coincides with $\fb\oplus\g_L$. Suppose that $\fb$ is irreducible and $\g_L$ is weakly irreducible.
	Then both distributions  $\mathcal{L}$ and $TN_0$ are horizontal.
	Since $\lambda\neq 0$, the geometry is indecomposable. 
		In the same time,  Lemma 3.6 from \cite{CMS21} implies that if a Riemannian manifold admits two orthogonal horizontal parallel   
		distributions, then the geometry is locally decomposable.
	 This shows a substantial difference between Riemannian and Lorentzian geometries with parallel skew-symmetric torsion.
	
}
\end{ex}

\section{Reducible case: $\dim L= 2,3$}\label{SecL=2,3}

\begin{theorem}\label{ThL=2}
	Let $(M, g, \nabla)$ be a Lorentzian geometry with a parallel skew-symmetric torsion $T$,  and let $\frakg\subset\so(1,n+1)$ be  the holonomy algebra of the connection $\nabla$ at a point $x\in M$. Suppose that the holonomy representation of $\frakg$ in $T_x M$ decomposes into a non-trivial orthogonal direct sum 
	\begin{equation}
	\label{tangentSpaceDecomposition2}
	T_x M = L \oplus E=\mathR^{1, 1}\oplus\Real^n
	\end{equation}
	 such that the induced representation of $\g$ in the  Lorentzian part $L=\mathR^{1, 1}$ is non-trivial, and $T\not\in\wedge^3E$.
	Then
	\begin{itemize}
		
		\item There exists a non-zero  vector $v\in E$, an orthogonal decomposition $E=\Real v\oplus E_1$, a Berger subalgebra $\fb_0\subset\so(E_1)$ with the torsion $\omega_{E_1}\in\wedge^3 E_1$, an endomorphism $\theta\in\wedge^2 E_1$ commuting with $\fb_0$ such that the torsion is given by 
				$$ T = p \wedge q \wedge v + \theta \wedge v + \omega_{E_1}. $$

		\item The holonomy algebra $\g$ is one of the following:		
		\begin{align*} \g_1 &= \mathR p\wedge q \oplus \fb, \\
\g_2 &= \{ \psi(B) p\wedge q + B \, | \, B \in \fb \},\end{align*}
		where the subalgebra $\fb\subset\so(E_1)$ is spanned by $\fb_0$ and $\theta$, and $\psi : \fb \to \mathR$ is a non-zero linear map with $\psi|_{[\fb,\fb]} = 0$.

		\item 
			Each algebraic curvature tensor $R\in\R^T(\g)$ is given by the following formulas:
		\begin{align*}
		& R(p,q) = \alpha p \wedge q - g(v,v) \theta, \\
		& R(X,Y) = g(v,v) \theta(X,Y) p \wedge q + C(X,Y), \quad X,Y\in E,\end{align*}
		where $$C(X,Y)=C_0(X,Y)+\theta(X,Y)\theta,\quad C_0\in\R^{\omega_{E_1}}(\fb_0),\quad \alpha\in\Real.$$
		In the case of the holonomy algebra $\g_2$ it holds $\alpha=-\psi(\theta)$ and $\psi(C(X,Y))= g(v,v) \theta(X,Y).$
			\end{itemize}
\end{theorem}

\begin{theorem}\label{ThL=3}
	Let $(M, g, \nabla)$ be a Lorentzian geometry with a parallel skew-symmetric torsion~$T$,   and let $\frakg\subset\so(1,n+1)$ be  the holonomy algebra of the connection $\nabla$ at a point $x\in M$. Suppose that the holonomy representation of $\frakg$ in $T_x M$ decomposes into a non-trivial orthogonal direct sum 
	\begin{equation}
	\label{tangentSpaceDecomposition1}
	T_x M = L \oplus E=\mathR^{1, 2}\oplus\Real^{n-1}
	\end{equation}
	such that the induced representation of $\g$ in the  Lorentzian part $L=\mathR^{1, 2}$ is weakly irreducible, and $T\not\in\wedge^3L\oplus\wedge^3E$.
	Then
	\begin{itemize}
		
		\item There exists a vector $v\in E$, a Berger subalgebra $\fb_0\subset\so(E_1)$ with the torsion $\omega_{E_1}\in\wedge^3 E_1$, where $E_1\subset E$ is the orthogonal complement to $v$ in $E$. Next, there exist a number $\alpha\in\Real$,  endomorphisms $\theta\in\wedge^2 E_1$, $\lambda\in\wedge^2 E$ commuting with $\fb_0$ such that the torsion is given by 
	$$T = p \wedge(\alpha  e_1 \wedge q +  e_1 \wedge v +  \lambda)+  \omega_{E_1}+\theta \wedge v.$$

			\item The holonomy algebra $\g$ is one of the following:		
	\begin{align*} 
		\g &= \mathR p\wedge e_1 \oplus \fb,\\
	\g &= \{ \psi(B) p\wedge e_1 + B \, | \, B \in \fb \},
	\end{align*}
	where $\fb\subset\so(E_1)$ is the subalgebra spanned by the endomorphisms $g(v,v)\theta+\alpha\lambda$ and $C_0(X,Y)+g(v,v)\theta(X,Y)\theta$ for all $C_0\in\R^{\omega_{E_1}}(\fb_0)$, $X,Y\in E$, and $\psi : \fb \to \mathR$ is a non-zero linear map with $\psi|_{[\fb,\fb]} = 0$.	
	
	\item Each algebraic curvature tensor $R\in\R^T(\g)$ is given by
	\begin{align*}
	 R(q,e_1) &= \beta p \wedge e_1 + g(v,v)\theta + \alpha \lambda, \\
	 R(X, Y) &= g(g(v,v)\theta X + \alpha\lambda X, Y)p \wedge e_1 + C(X, Y),  \quad X,Y\in E,
	\end{align*}
		where $$C(X,Y)=C_0(X,Y)+g(v,v)\theta(X,Y)\theta,\quad C_0\in\R^{\omega_{E_1}}(\fb_0),\quad \beta\in\Real.$$
In the case of the holonomy algebra $\g_2$ it holds $\beta=\psi(g(v,v)\theta + \alpha \lambda)$ and $\psi(C(X,Y))= g(g(v,v)\theta X + \alpha\lambda X, Y).$
	\end{itemize}
\end{theorem}

	The proofs of Theorems \ref{ThL=2} and  \ref{ThL=3} are direct and they use the techniques from the previous sections. For the proof of Theorem \ref{ThL=3} note that the projection of the torsion to $\wedge^3 L$ is proportional to the volume form.

\section{Summary}\label{SecSum}

Now we summarize the information about holonomy algebras form the previous sections. Let $(M,g,\nabla)$ be a Lorentzian geometry with non-zero parallel skew-symmetric torsion $T$ and holonomy algebra $\g\subset\so(1,n+1)$, then one of the following possibilities holds:

\begin{itemize}
	\item[1.] the holonomy algebra $\g\subset\so(1,n+1)$ is irreducible. This is possible only for $n=1$ and $\g=\so(1,2)$; in this case the torsion is proportional to the volume form.
	\item[2.] the holonomy algebra $\g\subset\so(1,n+1)$ is weakly irreducible and not irreducible. This situation is described in Section~\ref{Secpomega}. 
\end{itemize}

In other cases $\g$ preserves a decomposition of the tangent space
$$\Real^{1,n+1}=L\oplus E=\Real^{1,k+1}\oplus \Real^{n-k},\quad 1\leq\dim L= k+2\leq n-1$$
such that the induced representation of $\g$ in $L$ is weakly irreducible. In this case the geometry is reducible.  One of the possibilities is

\begin{itemize}
	\item[3.] The torsion at a point satisfies $T\in\wedge^3L \oplus \wedge^3 E$. Then the geometry is decomposable, i.e., locally it is a product of  Lorentzian and Riemannian geometries with parallel skew-symmetric torsion. In particular, $\g=\g_L\oplus\fb$, where $\g_L\subset\so(L)$, $\fb\subset\so(E)$ are the corresponding holonomy algebras.
\end{itemize}
	Now we may assume that $T\not\in\wedge^3L \oplus \wedge^3 E$, i.e., the geometry is not  decomposable. Then the following cases depending on the dimension of $L$ may appear:

\begin{itemize}
	\item[4.] $\dim L=1$; in this case $\g\subset\so(E)=\so(n+1)$, and the torsion is of the form $e_-\wedge\theta+\omega_E$, where $e_-$ is a  vector from $L$ of norm $-1$,  $\theta\in\wedge^2 E$, $\omega_E\in\wedge^3 E$, it holds $\g\cdot\theta=0$, $\g\cdot\omega_E=0$, $\theta\cdot\omega_E=0$; each $R\in\R^T(\g)$ is of the form $R=C_0-\theta\circ\theta$ for some $C_0\in\R^{\omega_E}(\g+\Real\theta)$. 
		\item[5.] $\dim L=2$; this case is considered in Section \ref{SecL=2,3};
	\item[6.] $\dim L=3$; this case is considered in Section \ref{SecL=2,3};
	\item[7.] $\dim L\geq 4$; this case is considered in Section \ref{SecdimL>3}.
\end{itemize}

\begin{cor}
	Let $(M,g,\nabla)$ be an indecomposable Lorentzian geometry with a parallel skew-symmetric torsion $T$ and holonomy algebra $\g\subset\so(1,n+1)$. Suppose that $M$ is simply connected.
	Then there exists a $\nabla$-parallel isotropic vector field in  the cases 2 with ($n\geq 2$), 6 and 7 from above. In the case 4 the vector field $p$ exists if and only if $\g$ annihilates a non-zero vector in $E$. The $\nabla$-parallel isotropic vector field $p$ is also parallel with respect to the Levi-Civita connection~$\nabla^g$. 
	\end{cor}

Thus in the most of the cases the geometry is locally defined by a Walker metric $g$ given by~\eqref{walkerMetric} with the parallel isotropic vector field $p=\partial_v$. The tensors determining the torsion may be considered as $\nabla$-parallel multivectors on the screen bundle $p^\bot/\left<p\right>$.

\section{Naturally reductive homogeneous spaces}\label{SecHom}

In Section \ref{SecPrel} we have seen that a connected and simply connected naturally reductive homogeneous space is uniquely determined by the corresponding infinitesimal model $(\m,R,T)$. Consequently, our task is to describe all infinitesimal models in the Lorentzian signature, i.e., for $\m=\Real^{1,n+1}$, $n\geq 1$.
Above we have found all holonomy algebras $\g\subset\so(1,n+1)$ of Lorentzian geometries $(M,g,\nabla)$ with parallel skew-symmetric torsion, we have described all torsions $T$ of these geometries, and all curvature tensors $R\in\R^T(\g)$. If such $\g$, $T$, and $R$ are fixed, then the second equality from \eqref{propRT1} and the equality 
\eqref{propRT2} hold true for $R$ and $T$. Consequently, 
$(\m,R,T)$ is an infinitesimal model with the holonomy algebra $\g$ if and only if $\g\cdot R=0$, $\g=\im R$, and \eqref{propRT3} holds true. Note that since $R$ satisfies \eqref{sympropR}, by \cite[Theorem A.2]{AFeF15}, the property \eqref{propRT3} holds true automatically. 
\emph{ In what follows we say that a    naturally reductive homogeneous space $(M,g)$ is indecomposable if the corresponding geometry $(M,g,\nabla)$, where $\nabla$ is the canonical connection, is indecomposable in the sense of the definition from Section~\ref{SecRed}.}

Let us give some constructions of infinitesimal models in Lorentzian signature.

\begin{ex}\label{ConstrdimL>3} {\rm
		Let $(E_1,C_0,\omega_{E_1})$ be an infinitesimal model of Riemannian signature. Let $\n\subset\so(E_1)$ be a commutative subalgebra commuting with $\fb_0=\im C_0$ and annihilating $C_0$ and $\omega_{E_1}$. Choose an Euclidean metric on $\n$.  Let $\dim \n=l$ and let $\Real^l$ and $E_0$ be two copies of $\n$. 
	Consider the Minkowski space
	$$\Real^{1,n+1}=L\oplus E,$$
	where $$L=\Real p\oplus\Real^k\oplus\Real q,\quad \Real^k=\Real^{k-l}\oplus\Real^l,\quad k\geq 2,$$
	$$E=E_1\oplus E_0.$$
	
	Let $\omega_{\Real^k}\in\wedge^2\Real^k$ be an arbitrary element.
	Let $\varphi\in\n\wedge E_0$ and $\zeta_2\in\Real^l\wedge E_0$ correspond to the identity maps under the identifications $E_0\cong\n$ and $\Real^l\cong E_0$, respectively.
	Define the 3-form $$\omega_E=\omega_{E_1}+\varphi\in\wedge^3 E$$
Let $\lambda\in\wedge^2E$ be an element such that	$\lambda	\cdot \omega_{E}=0$. 		Let $\zeta_1\in\Real^k\wedge E_1$ be an element satisfying the conditions $\fb\cdot \zeta_1=0$, $\omega_{E_1}(\zeta_1(\Real^k))=0$, and $\zeta_1(E_1)\cap\zeta_2(E_0)=0$. 	
	Let 	
	$$T=p\wedge \zeta+\omega_E,$$ where 
		$$\zeta=\omega_{\Real_k}+\zeta_1+\zeta_2+\lambda.$$
	 Let $V_1,\dots,V_l$ be an orthogonal basis of $E_0$. Let $\theta_i=\varphi(V_i)$, and $X_i=\zeta_2(V_i)$.
	Let $K:\Real^n\to\Real^n$ be a symmetric linear map such that $\im K$ and $\zeta_2(E_0)$ span $\Real^k$.
	Define the tensor $R$  with the following non-zero values:
	\begin{align*} R(q,X)&=p\wedge K(X)+\sum_{i = 1}^l g(X,X_i)\theta_i,\quad X\in\Real^k,\\
	R(Y,Z)&= p\wedge \sum_{i = 1}^l\theta_i(Y,Z)X_i+C_0(Y,Z)+\sum_{i = 1}^l\theta_i(Y,Z)\theta_i,\quad Y,Z\in E_1. 
	\end{align*} 
	It is easy to check that $(\Real^{1,n+1},R,T)$ is an infinitesimal model. The projection of the holonomy algebra $\g=\im R$ to $\so(L)$ coincides with $p\wedge\Real^k$.
}
		\end{ex}

\begin{ex}\label{ConstrdimL=3}{\rm Let
$(E_1,C_0,\omega_{E_1})$ be an infinitesimal model of Riemannian signature. Let $\Real$ be the line with the scalar product, and $v\in\Real$ an arbitrary (possibly zero) vector. Consider the Euclidean space $E=E_1\oplus\Real v$. Let $n=\dim E+1$.
Consider the Minkowski space
$$\Real^{1,n+1}=L\oplus E=\Real^{1,2}\oplus E.$$
Fix  endomorphisms $\theta\in\wedge^2 E_1$, $\lambda\in\wedge^2 E$ commuting with $\im C_0$. Let $\alpha,\beta\in\Real$.
Define the algebraic torsion		$$T = p \wedge(\alpha  e_1 \wedge q +  e_1 \wedge v +  \lambda)+  \omega_{E_1}+\theta \wedge v$$
and curvature
\begin{align*}
	R(q,e_1) =& \beta p \wedge e_1 + g(v,v)\theta + \alpha \lambda, \\
		R(X, Y) =& g(g(v,v)\theta X + \alpha\lambda X, Y)p \wedge e_1 \\&+ C_0(X,Y)+g(v,v)\theta(X,Y)\theta,\quad X,Y\in E.  
		\end{align*}	
	It is easy to check that $(\Real^{1,n+1},R,T)$ is an infinitesimal model. }
\end{ex}

\begin{ex}\label{ConstrdimL=2}{\rm Let
		$(E_1,C_0,\omega_{E_1})$ be an infinitesimal model of Riemannian signature, $\dim E_1=n-1$. Let $\Real$ be the line with the scalar product, and $v\in\Real$ an arbitrary non-zero vector. 
		Consider the Minkowski space
		$$\Real^{1,n+1}=L\oplus E=\Real^{1,1}\oplus (E_1\oplus\Real v).$$
		Fix an endomorphism $\theta\in\wedge^2 E_1$ commuting with $\im C_0$. Let $\alpha\in\Real$.
		Define the algebraic torsion		$$T = p \wedge q \wedge v + \theta \wedge v + \omega_{E_1}$$
		and  curvature
		\begin{align*}
		& R(p,q) = \alpha p \wedge q - g(v,v) \theta, \\
		& R(X,Y) = g(v,v) \theta(X,Y) p \wedge q + C_0(X,Y)+\theta(X,Y)\theta, \quad X,Y\in E.\end{align*}
			We get an infinitesimal model  $(\Real^{1,n+1},R,T)$. 
		}
		\end{ex}

\begin{ex}\label{ConstrdimL=1}{\rm Let $(E=\Real^{n+1},C_0,\omega_E)$ be an infinitesimal model of Riemannian signature. Suppose that a $\theta\in\wedge^2 E$ is given, and it holds $\theta\cdot C_0=0$, $\theta\cdot \omega_E=0$. Consider the Minkowski space
$$\Real^{1,n+1}=L\oplus E=\Real e_-\oplus \Real^{n+1},$$
where $e_-$ is a vector of norm $-1$. Define the tensors
$$T=e_-\wedge\theta+\omega_E,$$
and $$R=C_0-\theta\circ\theta.$$
		We get an infinitesimal model  $(\Real^{1,n+1},R,T)$. 
}
\end{ex}

\begin{theorem}\label{Thnrhs} Let $(M,g)$ be a connected and simply connected indecomposable  naturally reductive homogeneous Lorentzian space of dimension $n+2\geq 4$. Then either $(M,g)$ is a symmetric space or one of the following holds:
	\begin{itemize}
			\item $(M,g)$ is $\Real^{n+2}$ with the structure of a homogeneous plane wave described above in Example~4.
				
				\item The infinitesimal model $(\Real^{1,n+1},R,T)$ of  $(M,g)$ is given by one of  Examples \ref{ConstrdimL>3}--\ref{ConstrdimL=1}.
					\end{itemize}
	\end{theorem}
\begin{proof}  Let $(M,g)$ be a connected and simply connected naturally reductive homogeneous Lorentzian space. Let $\nabla$ be the canonical connection, and let $R$ be the curvature tensor of $\nabla$.
	If $R=0$, then by \eqref{dT}, $dT=0$, and as we have seen  in Example~3 above, $(M,g)$ is a symmetric space. Hence we may assume that the holonomy algebra $\g$ of the canonical connection $\nabla$ is non-trivial. 	
	
	We consider case by case the holonomy algebras from the list given in Section~\ref{SecSum}. Consider the holonomy algebra from the case	{\bf  2.} In this case $\g\subset\so(1,n+1)$ is weakly irreducible and not irreducible.  	
	Since $\dim M\geq 4$,
	 we are in the settings  of Section~\ref{Secpomega}. In Section~\ref{Secpomega} we have seen that each algebraic curvature tensor $R$ is determined by the components $R_1$, $P$, and $K$.
	The condition $\g\cdot R=0$ immediately implies $R_1=0$  and $P=0$.
This and equality \eqref{curvatureTensorOmega} imply that $g$ is a pp-wave metric, i.e., $$ g = 2 dv du + \sum_{i=1}^n (dx^i)^2 + H (du)^2, $$
where $H=H(x^1,\dots x^n,u)$. The torsion may be written
as $$T=du\wedge\omega,\quad\omega=2\sum_{1\leq i<j\leq n}\omega_{ij}dx^i\wedge dx^j,$$
where $\omega_{ij}$ are functions. The condition $\nabla T=0$ easily implies that $\omega_{ij}$ are constants (in Section \ref{Secpomega} we have seen that $\omega$ is a parallel section of the screen bundle, which is now flat). The equality \eqref{curvatureTensorOmega} shows that the only non-zero values of the curvature tensor are
$$R(\partial_u,X)=p\wedge K(X)=R^g(\partial_u,X)-\frac{1}{4}p\wedge \omega^2(X),$$ where $X$ is a combination of the vector fields $\partial_{x^i}$, $i=1,\dots,n$. Next, $R^g$ is determined by
$$R^g(\partial_u,X)=p\wedge K_0(X),$$
$$K_0=\frac{1}{2}\sum_{ i,j=1 }^n(\partial_{x^i}\partial_{x^j}H) \partial_{x^i}\otimes dx^j,$$
see, e.g., \cite{GalRMS}. The condition $\nabla R=0$ may be rewritten as the condition $\nabla^EK=0$.
We get that $$\nabla^E_{\partial_{x^i}}K=\nabla^E_{\partial_{x^i}}K_0=\nabla^{g,E}_{\partial_{x^i}}K_0.$$
This implies that $\partial_{x^l}\partial_{x^i}\partial_{x^j}H=0$, i.e.,
$H=\sum_{i,j=1}^nH_{ij}(u)x^ix^j$ (the terms linear in $x^i$ may be omitted without loss of generality). Next, 
$$\nabla^E_{\partial_{u}}K=\nabla^E_{\partial_{u}}K_0=\left(\nabla^{g,E}_{\partial_{u}}+\frac{1}{2}\omega \right)K_0=\nabla^{g,E}_{\partial_{u}}K_0-[F,K_0],$$
where $F$ is a linear map with the matrix $F^i_j=\omega_{ij}$.
This shows that the matrix $K_0$ satisfies
$$\partial_u K_0=[F,K_0].$$
The solution of this system of equations
may be obtained in the form
$$K_0=(e^{-uF})^\intercal B e^{-uF}$$
for a constant matrix $B$. Thus we arrive to the settings of Example~4. The space $(M,g)$ is symmetric if and only if 
$\nabla^{g,E}K_0=0$, which is equivalent to the condition $[F,K_0]=0$.

Since we consider indecomposable geometry $(M,g,\nabla)$, its holonomy algebra and torsion cannot be as in the case {\bf 3}.
If the holonomy algebra is as in the case {\bf 4}, we immediately see the infinitesimal model of $(M,g)$ is as in Example \ref{ConstrdimL=1}. In the cases {\bf 5} and {\bf 6} it is easy to check that the infinitesimal model is as in Examples \ref{ConstrdimL=2} and \ref{ConstrdimL=3}, respectively.
 	Suppose that	the holonomy algebra is as in the case {\bf 7}. Consider $T$ and $R$ as in Theorem \ref{SecdimL>3}. The condition $p\wedge\Real^{k_1}\cdot R=0$ easily implies the equalities $R_0=0$ and $P=0$. This shows that the infinitesimal model is as in Example \ref{ConstrdimL>3}.	
\end{proof}

Now we consider spaces of dimension 3, 4, and 5. The simply connected naturally reductive homogeneous spaces $(M,g)$ are assumed to be indecomposable.

{\bf Dimension of $M$ is $3$.} Connected and simply connected naturally reductive homogeneous Lorentzian space of dimension~3 were classified in \cite{CM08,HBRR}, where it is shown that these spaces   are exhausted by symmetric spaces, Lie groups $\SU(2,\Real)$, $\widetilde{\SL(2,\Real)}$
and Heisenberg group with  suitable left-invariant metrics. We prove this result by our method. 

Since $\dim M=3$, the torsion $T$ is proportional to the volume form. Making a rescaling, we may assume that the torsion coincides with the volume form. We will consider a Witt basis $p,e,q$ of $\Real^{1,2}$ and assume that $T=p\wedge e\wedge q$. It is easy to see that the right-hand side of the Bianchi identity is zero. Hence each algebraic curvature tensor with torsion $T$ is  just an algebraic curvature tensor with zero torsion. 
 
	Suppose that $\g=\so(1,2)$. 
	 Each invariant algebraic curvature is of the form $R(X,Y)=c X\wedge Y$, where $c$ is a non-zero constant. Since $c\neq 0$, the reductive decomposition
$\so(1,2)\oplus\Real^{1,2}$ of the isometry algebra may be chosen in such a way that $[\Real^{1,2},\Real^{1,2}]\subset \so(1,2)$, i.e., $(M,g)$ is a symmetric space.

Suppose that the holonomy algebra $\g$ is weakly irreducible and not irreducible. Then $\g$ is either $\Real p\wedge q\zr\Real p\wedge e$ or $\Real p\wedge e$. The first algebra does not admit a suitable algebraic curvature tensor. Thus, $\g=p\wedge e$. The curvature tensor is defined by the equality 
$$R(q,e)=\alpha p\wedge e$$ for some non-zero $\alpha$. 
The algebra $\frakf$ has dimension 4 and it holds 
\begin{align*}
&[p, q] = -e, \quad [p,e] = p, \quad
[e, q] = q + \alpha p\wedge e, \\
&[p \wedge e, q] = e, \quad [p \wedge e,e] = -p.  
\end{align*}
The derived algebra is given by
$$\frakf^\prime = \left< p, e, q + \alpha p\wedge e \right>$$
and it is \emph{transversal} to $\g$. 
Let  $A = p$, $B = e$, $C = q + \alpha p\wedge e $. Then it holds that
$$[A, B] = A, \quad [A, C] = -B, \quad [B, C] = \alpha A + C.$$
The Killing form of $\f'$ with respect to basis $(A,B,C)$ is given by the matrix
$$ K = \left(\begin{array}{ccc}
0 & 0 & 2 \\
0 & 2 & 0 \\
2 & 0 & -2\alpha
\end{array}\right). $$
Thus $\frakf^\prime$ is isomorphic to $\so(1,2)$, and $(M,g)$ is isometric to the Lie group $\widetilde{\SL(2,\Real)}$
 with a suitable left-invariant metric.

Suppose that $\g$ preserves the orthogonal decomposition
	$\mathR^{1,2} = \mathR^{1, 0} \oplus \mathR^{2}.$
	Fix an orthonormal basis $(e_1, e_2)$ in $\mathR^{2}$ and let $e_-$ be a unit time-like vector in $\mathR^{1, 0}$.
	The holonomy algebra is one-dimensional, 
	$$ \g =\so(2)= \mathR e_1 \wedge e_2. $$
	The torsion and the curvature are given by 
	$$ T =  e_- \wedge e_1 \wedge e_2, \quad  R(e_1, e_2) = \beta e_1 \wedge e_2, $$
	where $\beta \in \mathR$ is non-zero.
	The algebra $\frakf$ has dimension 4 and it holds 
	\begin{align*}
	&[e_-, e_1] =  e_2, \quad [e_-, e_2] = - e_1,  \quad
	[e_1, e_2] = - e_- - \beta e_1 \wedge e_2, \\
	&[e_1 \wedge e_2, e_1] = e_2, \quad [e_1 \wedge e_2, e_2] = -e_1.
	\end{align*}
	The derived algebra is given by
	$$\frakf^\prime = \left< e_1, e_2, \alpha e_- + \beta e_1 \wedge e_2 \right>$$
	and it is \emph{transversal} to $\g$.
	Let  $A = e_1$, $B = e_2$, $C =  e_- + \beta e_1 \wedge e_2$. Then it holds that
	$$[A, B] = C, \quad [A, C] = -\gamma B, \quad [B, C] = \gamma A, 
$$
	where $\gamma = 1 + \beta$.
	The Killing form with respect to basis $(A,B,C)$ is given by the matrix
	$$ K = \left(\begin{array}{ccc}
	-2\gamma & 0 & 0 \\
	0 & -2\gamma & 0 \\
	0 & 0 & -2\gamma^2
	\end{array}\right), $$
	thus $\frakf^\prime$ is isomorphic to the Heisenberg algebra $\h_3$ if $\beta=-1$, $\frakf^\prime$ is isomorphic to $\so(3)$ if $\beta<-1$ and it is isomorphic to $\so(1,2)$ otherwise.

	Suppose that $\g$ preserves the decomposition 	$\mathR^{1, 1} = \mathR^{1, 1} \oplus \mathR^{1}$.
	Let $p,q$ be a Witt basis of $\mathR^{1, 1}$. The holonomy algebra is one-dimensional,
	$$ \g = \mathR p \wedge q. $$ 
	The torsion and the curvature are given by 
	$$ T = p \wedge q \wedge e,\quad R(p, q) = \beta p \wedge q, $$
	where  $e \in \mathR^{1}$ is a unite vector and  $\beta \in \mathR$ is non-zero.
	The Lie algebra $\frakf$ is of dimension 4 and it holds 
	\begin{align*}
	&[p, q] = e - \beta p \wedge q, \quad
	[e, p] =  p, \quad [e,q] = - q, \\
	&[p \wedge q, p] = -p, \quad [p \wedge q,q] = q. 
	\end{align*}
	The derived algebra is given by
	$$\frakf^\prime = \left< p, q, e - \beta p \wedge q \right>$$
	and it is \emph{transversal} to $\g$. Let $A = p$, $B = q$, $C = e - \beta p \wedge q$. Then it holds that
	$$[A, B] = C, \quad
	[A, C] = -\gamma A, \quad
	[B, C] = \gamma B, 
	$$
	where $\gamma = 1 + \beta$.
	The Killing form of $\f'$ with respect to the basis $(A,B,C)$ is given by the matrix
	$$ K = \left(\begin{array}{ccc}
	0 & 2\gamma & 0 \\
	2\gamma & 0 & 0 \\
	0 & 0 & 2\gamma^2
	\end{array}\right). $$
	Thus $\frakf^\prime$ is isomorphic to the Heisenberg algebra $\h_3$ if $\beta = -1$ and to $\so(1,2)$ otherwise.
	
	\medskip
	
{\bf Dimension of $M$ is $4$.} Connected simply connected homogeneous Lorentzian spaces of dimension 4 were classified in~\cite{BLM15}. We apply our method and prove the following theorem.

\begin{theorem}\label{Thhomdim4} Let $(M,g)$ be a connected and simply connected indecomposable  naturally reductive homogeneous Lorentzian space of dimension $4$. Then   either $(M,g)$ is a symmetric space or one of the following holds:
	\begin{itemize}
		\item $(M,g)$ is isometric to the homogeneous space $F/G$ with an invariant metric, where  $F=\widetilde{\SL(2,\Real)}\times\Real^2$, and $G\subset F$ is a one-dimensional subgroup;
		\item $(M,g)$ is isometric to the homogeneous space $F/G$ with an invariant metric, where  $F=\SU(2,\Real)\times\Real^2$, and $G\subset F$ is a one-dimensional subgroup;
				\item $(M,g)$  is $\Real^{4}$ with the structure of a homogeneous plane wave described above in Example~4.
	\end{itemize}
	\end{theorem}

\begin{rem} {\rm In  the classification Theorem \cite[Theorem 9]{BLM15}  is given the transvection algebra 
	with the basis $\{Y_1, Y_2, Y_3, T_1, T_2\}$ and the only non-zero brackets
	$$[Y_1, Y_2] = -\lambda Y_3,\quad  [Y_1,Y_3] = \lambda Y_2,\quad [Y_2,Y_3] = Y_1.$$ Unfortunately, it was wrongly stated that the Lie algebra $\left<Y_1, Y_2, Y_3 \right>$ is isomorphic to $\mathfrak{sl}(2,\Real)$. 	Analyzing the Killing form of the Lie algebra $\left<Y_1, Y_2, Y_3 \right>$, it is easy to see that  this Lie algebra is isomorphic to $\mathfrak{sl}(2,\Real)$ if $\lambda>0$  and it is isomorphic to $\su(2)$ if $\lambda<0$. 
	The 	third space from Theorem~\ref{Thhomdim4}  coincides with the second space from \cite[Theorem 9]{BLM15}, we explain this in the proof below. Thus the classifications in Theorem~\ref{Thhomdim4} and \cite[Theorem 9]{BLM15}
are essentially the same.
} 	
\end{rem}

\begin{proof} 	Consider the canonical connection $\nabla$. Let $\g\subset\so(1,3)$ be the holonomy algebra of~$\nabla$.
	First suppose that  $\g\subset\so(1,3)$ is weakly irreducible. Then by Theorem \ref{Thnrhs}, $\g = p \wedge \mathR^2 $.
	The torsion and the curvature are given by
	$$ T = p \wedge e_1 \wedge e_2,\quad  R(q,X) = p \wedge K(X), $$
	where $K$ is a symmetric  endomorphism of $\mathR^2$. We may assume that $K = \operatorname{diag}(\lambda_1, \lambda_2)$. The indecomposability implies that  $\lambda_1$ and  $\lambda_2$ are non-zero.
	The algebra $\frakf$ has dimension 6 and 
	\begin{align*}
	&[p, \frakf] = 0, \quad [e_1, e_2] = -p, \\
	&[e_1, q] =  e_2 + \lambda_1 p \wedge e_1, \quad [e_2,q] = -e_1 + \lambda_2 p \wedge e_2, \\
	&[p \wedge e_i, q] = e_i, \quad [p \wedge e_i,e_j] = -\delta_{ij}p.  
	\end{align*} The derived algebras are given by
	$$\frakf^\prime = \left< p, e_1, e_2, p\wedge e_1, p\wedge e_2 \right>,\quad \f''=\left<p\right>,$$ i.e., $\f$ is solvable.
		The space  corresponds to the case 2 of  \cite[Theorem 9]{BLM15}. We have just reduced the number of the possible parameters from 4 to 2 by a proper choice of the basis:
		we assumed that  $T = c p \wedge e_1 \wedge e_2$ with $c=1$ and we diagonalized the symmetric endomorphism $K$.

	Suppose that $\g$ preserves the decomposition
	$ \mathR^{1, 3} = \mathR^{1, 0} \oplus \mathR^3$. We are in the settings of Example~\ref{ConstrdimL=3}.
	Since $T \neq 0$,  $\g$ is a proper subalgebra in $\so(3)$. Therefore we can assume that $\g = \mathR e_1 \wedge e_2$ 
	and the torsion is given by 
	$$ T = e_- \wedge e_1 \wedge e_2 + \beta e_1 \wedge e_2 \wedge e_3 = e_1 \wedge e_2 \wedge (e_- + \beta e_3),$$
	where $\beta \in \mathR$ is non-zero.
	If $\beta \ne \pm 1$,   then the vector $\beta e_- +  e_3$ is non-isotropic, it is annihilated by the holonomy algebra and the torsion, i.e., the geometry is decomposable.
	We may assume that $\beta = 1$.
	The curvature tensor is given by 
	$$ R(e_1, e_2) = \gamma e_1 \wedge e_2, $$
	where $\gamma \in \mathR$ is non-zero.
	The algebra $\frakf$ has dimension 5 and 
	\begin{align*}
	&[p, \frakf] = 0, \quad [e_1, e_2] = -p - \gamma e_1 \wedge e_2, \\
	&[e_1, q] = 2 e_2, \quad [e_2,q] = -2 e_1.
	\end{align*}
The	center of $\frakf$ is given by 
	$$ \mathfrak{z}(\frakf) = \left< p, q + 2 e_1 \wedge e_2 \right>, $$
	and we have 
	$$ \frakf^\prime = \left< e_1, e_2, p + \gamma e_1 \wedge e_2 \right>, $$
	which is isomorphic to $\so(1,2)$ if $\gamma > 0$ and it is isomorphic  to $\so(3)$ if $\gamma < 0$. Thus $\f$ is either $\so(1,2)\oplus\Real^2$ or $\so(3)\oplus\Real^2$.
	This  space corresponds to the case 1 of Theorem 9 from \cite{BLM15}.

	If $\g$ preserves the decomposition
	$ \mathR^{1, 3} = \mathR^{1, 1} \oplus \mathR^2, $
	then according to Example~\ref{ConstrdimL=2} it holds
	$$ T = p \wedge q \wedge v, $$
	where $v\in\Real^ 2$. Since $\g$ annihilates $v$, it holds $\g=\Real p\wedge q$. Let $v_1\in\Real^2$ be a non-zero vector orthogonal to $v$. Then $v_1$ is annihilated be the holonomy algebra and by the torsion, i.e., the space is  decomposable.

	Finally, if $\g$ preserves the decomposition
	$ \mathR^{1, 3} = \mathR^{1, 2} \oplus \mathR^1, $
	then according to Example \ref{ConstrdimL=3} it holds that
	$$ T = \alpha p \wedge e_1 \wedge q + p \wedge e_1 \wedge v=p \wedge e_1 \wedge v\wedge (\alpha q+v). $$
The vector  $ \alpha v - g(v,v)p$ is space-like, it is annihilated be the holonomy algebra and by the torsion, i.e., the space is again decomposable.

\end{proof}

\medskip

{\bf Dimension of $M$ is $5$.}  To our knowledge, the following result is new.

\begin{theorem} Let $(M,g)$ be a connected and simply connected indecomposable  naturally reductive homogeneous Lorentzian space of dimension $5$. Then either $(M,g)$ is a symmetric space or one of the following holds:
	\begin{itemize}
		\item $(M,g)$ is $\Real^{5}$ with the structure of a homogeneous plane wave described above in Example~4;
		
		\item $(M,g)$ is isometric to a Berger sphere $\SU(3)/\SU(2)$, or $\SU(1,2)/\SU(2)$;

\item		$(M,g)$ is isometric to the homogeneous space $(F_1\times F_2)/G$  with an invariant metric, where each of $F_1$ and $F_2$ are isomorphic to one of the Lie groups $\widetilde{\SL(2,\Real)}$, $\SU(2)$, and $G$ is either $\SO(2)$ or $\SO(1,1)$;

\item $(M,g)$ is isometric to one of the following homogeneous spaces endowed with an invariant metric: $(\widetilde{\SL(2,\Real)}\times H_3)/G$, where $G=\SO(1,1)\subset \widetilde{\SL(2,\Real)}$;  $(H_3\times\SU(2))/G$, where $G=\SO(2)\subset \SU(2)$;
$(H_3\times \widetilde{\SL(2,\Real)})/G$, where $G=\SO(1,1)\subset \widetilde{\SL(2,\Real)}$;

\item		$(M,g)$ is isometric to the Heisenberg group $H_5$  with a left-invariant metric.
\end{itemize}

\begin{proof} Consider the canonical connection $\nabla$. Let $\g\subset\so(1,4)$ be the holonomy algebra of~$\nabla$.
	If the holonomy algebra $\g\subset\so(1,4)$ is weakly irreducible, then by Theorem \ref{Thnrhs}, $(M,g)$ is a regular homogeneous plane wave. Next we consider verschiedene $\g$-invariant decompositions $\Real^{1,4}=\Real^{1,r}\oplus\Real^{4-r}$, $0\leq r\leq 3$ and assume that the induced representation of $\g$ in $\Real^{1,r}$ is weakly irreducible.
		
	Suppose that the holonomy algebra $\g$ preserves the decomposition $\Real^{1,4}=\Real^{1,0}\oplus\Real^4$, i.e., we are in the situation of Example~\ref{ConstrdimL=1}. We claim that the form $\omega$ is zero. Indeed, suppose that $\omega\neq 0$. Then $\g$ annihilates the vector $*\omega\in\Real^4$. We may assume that $*\omega=c e_4$, $\omega=ce_1\wedge e_2\wedge e_3$. Then, $\g\subset\so(3)$. Since $\g$ commutes with $\theta$, $\g$ is a proper subalgebra of $\so(3)$, i.e.,
	we may assume that $\g=\Real e_1\wedge e_2$. We see that $\theta=c_1 e_1\wedge e_2+c_2 e_3\wedge e_4$. The condition $\theta\cdot \omega=0$ implies $c_2=0$. It is clear that $(M,g)$ is decomposable, i.e., we get a contradiction. Thus, $\omega=0$.
	
	Consider now the following decomposition of the Lie algebra
	$\f$:
	$$\f=(\g\oplus \Real e_-)\oplus \Real^4.$$
	It is clear that this decomposition defines a $\mathbb{Z}_2$-grading of $\f$. It holds that $${\rm ad}_{e_-}|_{\Real^4}=\theta.$$
	 The Lie bracket restricted to $\Real^4$ satisfy
	\begin{equation}\label{odnaskLie}[X,Y]=-C_0(X,Y)+\theta(X,Y)\theta-\theta(X,Y)e_-.\end{equation}
	The algebraic curvature tensor $C_0$ defines the following $\mathbb{Z}_2$-graded Lie algebra:	$$\h\oplus\Real^4,$$ where $\h=C_0(\Real^4,\Real^4)\subset\so(4)$. Since $\h$ commutes with $\theta$, we see that $\h$ is one of the following Lie algebras: $\fu(2)$, $\so(2)\oplus\so(2)$, $\so(2)$, or it is trivial. Let us consider these cases.
	
	Consider the case $\h=\fu(2)$. The tensor $C_0$ is given by
	$$C_0(X,Y)=a\big(X\wedge Y+JX\wedge JY+2(JX,Y)J\big),$$ where $J$ is the complex structure on $\Real^4$, and $a\neq 0$. It is clear that $\theta=b J$ for some $b\in\Real$. Recall that $\g=R(\Real^{1,4},\Real^{1,4})=R(\Real^4,\Real^4)$. It holds that 
	$\g=\su(2)$ if and only if $b^2-3a=0$. Otherwise, $\g=\fu(2)$. 
	If $\g=\su(2)$, then it is obvious that 
	$$\f=(\su(2)\oplus \Real e_-)\oplus \Real^4$$ is isomorphic either to $\su(3)$ or to $\su(1,2)$ depending on the sign of $a$. Consequently $(M,g)$ is isometric to one of the homogeneous spaces $\SU(3)/\SU(2)$, $\SU(1,2)/\SU(2)$. If $\g=\fu(2)$, then $\f$ is isomorphic to one of the Lie algebras $\Real(\theta-e_-)\oplus \su(3)$, 	$\Real(\theta-e_-)\oplus \su(1,2)$, and $(M,g)$ is again isometric to one of the homogeneous spaces $\SU(3)/\SU(2)$, $\SU(1,2)/\SU(2)$.

	Suppose that $\h=\so(2)\oplus\so(2)=\Real e_1\wedge e_2\oplus\Real e_3\wedge e_4$. Then \begin{equation}\label{odnotheta}\theta=c_1 e_1\wedge e_2+c_2 e_3\wedge e_4.\end{equation}
	The curvature tensor $C_0$ is given by
	$$C_0(e_1,e_2)=ae_1\wedge e_2,\quad C_0(e_3,e_4)=be_3\wedge e_4$$
for some $a,b\neq 0$. 
We get that $$R(e_1,e_2)=(a-c_1^2)e_1\wedge e_2-c_1c_2e_3\wedge e_4,\quad R(e_3,e_4)=-c_1c_2e_1\wedge e_2+(b-c_2^2) e_3\wedge e_4.$$
The Lie bracket of the Lie algebra $\f$ satisfy
$$[e_1,e_2]=c_1(\theta-e_-)-ae_1\wedge e_2,\quad [e_3,e_4]=c_2(\theta-e_-)-be_3\wedge e_4.$$
The indecomposability implies $c_1c_2\neq 0$.
We conclude that $\g$ is one-dimensional if and only if $ab-ac_2^2-bc_1^2=0$. In that case $$\g=\Real \xi,\quad \xi=((a-c_1^2)e_1\wedge e_2-c_1c_2e_3\wedge e_4).$$
Consequently we get 
$$\f=[\Real^4,\Real^4]\oplus\Real^4.$$ This Lie algebra isomorphic to $\f_1\oplus \f_2$, where $\f_1$ and $\f_2$ are isomorphic to one of the Lie algebra $\so(3)$ or $\so(1,2)$ (depending on the signs of $a$ and $b$). The subalgebra $\g=\so(2)\subset\f_1\oplus \f_2$ is included diagonally. If $ab-ac_2^2-bc_1^2\neq 0$, then $\g=\so(2)\oplus\so(2)$ and 
$$\f=\big([\Real^4,\Real^4]\oplus\Real (\theta-e_-)\big)\oplus\Real^4.$$ In that case $\f$ contains the center $\Real(\theta-e_-)$. The Lie algebra $[\Real^4,\Real^4]\oplus\Real^4$ is again isomorphic to $\f_1\oplus \f_2$. 
The projection of this Lie algebra to $\Real^{1,4}=\Real e_-\oplus\Real^4$ coincides with
$\Real^{1,4}$, i.e., the corresponding connected Lie subgroup of $F$ acts transitively on $M$.

Suppose that $\h=\so(2)=\Real e_1\wedge e_2$. Then $\theta$ is again given by \eqref{odnotheta}.
Now $C_0$ is given as in the previous case with $b=0$ and $a\neq 0$.
As above, $c_1c_2\neq 0$. We see that $\g$ is 2-dimensional.
In that case $$\f=\big([\Real^4,\Real^4]\oplus\Real e_3\wedge e_4\big)\oplus\Real^4,$$ and   
$[\Real^4,\Real^4]\oplus\Real^4$ is isomorphic to $\f_1\oplus\h_3$, where $\f_1$ is as above, and $\h_3$ is the Heisenberg algebra.

Suppose that $\h=0$, i.e., $C_0=0$. We again may assume that  
	$\theta$ is  given by \eqref{odnotheta}. We get that $\g=\Real\theta$. The indecomposability implies $c_1c_2\neq 0$. The subalgebra $\Real(\theta-e_-)\oplus\Real^4\subset\f$ is transversal to $\g$ and it is isomorphic to the Heisenberg algebra. Consequently $(M,g)$ is isometric to the Heisenberg group with a left-invariant metric.

	Suppose that $\g$ preserves the decomposition $\Real^{1,4}=\Real^{1,1}\oplus\Real^3$. According to Example~\ref{ConstrdimL=2}, $$T=p\wedge q\wedge v+\theta\wedge v,$$
	where $v\in\Real^3$ is a non-zero vector, and $\theta\in\wedge^2\Real^3$ with $\theta(v)=0$. We may assume that the norm of $v$ is 1. Let $e_1,e_2,v$ be an orthonormal basis of $\Real^3$ such that $\theta=a e_1\wedge e_2$. We get that
	$$T=\eta\wedge v,\quad \eta=p\wedge q+a e_1\wedge e_2$$
	and $$R=C_0+\eta\circ\eta,\quad C_0\in\R^0(\g+\Real\eta).$$
	The Lie algebra $\f$ admits the following $\mathbb{Z}_2$-grading:
	$$\f=(\g\oplus\Real v)\oplus\Real^{1,3},$$
	where $\Real^{1,3}=\left<p,q,e_1,e_2\right>$. The Lie bracket
	restricted to   $\Real^{1,3}$ satisfies
	$$[X,Y]=-R(X,Y)-\theta(X,Y)v=-C_0(X,Y)-\eta(X,Y)\eta-\eta(X,Y)v.$$
	Next, $$\textrm{ad}_v|_{\Real^{1,3}}=-\eta.$$
	
	We see that $C_0$ defines the $\mathbb{Z}_2$-graded Lie algebra
	$$\h\oplus\Real^{1,3},\quad \h=C_0(\Real^{1,3},\Real^{1,3}).$$
	It is clear that $\h\subset\Real p\wedge q\oplus\Real e_1\wedge e_2$,
	and $C_0$ is given by
	$$C_0(p\wedge q)=b p\wedge q, \quad C_0(e_1,e_2)=c e_1\wedge e_2.$$
	The Lie algebra $\h$ is either one of the following: $\so(1,1)\oplus\so(2)$,
	$\so(1,1)$, $\so(2)$, or it is trivial.
		For the Lie bracket of $\f$ it holds
	$$[p,q]=-bp\wedge q+\eta+v,\quad [e_1,e_2]=-ce_1\wedge e_2-a\eta-av.$$ 	The rest of the considerations is as above. Suppose that $\h=\so(1,1)\oplus\so(2)$, then $b,c\neq 0$. It holds that $\dim\g=1$ 
	 if and only if $a^2b+cb-c=0$. Otherwise, $\dim\g=2$. 
	 If $\dim\g=1$, then $$\f=[\Real^{1,3},\Real^{1,3}]\oplus\Real^{1,3}=\so(1,2)\oplus \f_2,$$ where $\f_2$ is either $\so(3)$ or $\so(1,2)$.
	 If $\dim\g=2$, then $$\f=\big([\Real^{1,3},\Real^{1,3}]\oplus\Real(\eta+v)\big)\oplus\Real^{1,3}=\so(1,2)\oplus \f_2,$$ where $\f_2$ is again either $\so(3)$ or $\so(1,2)$. Suppose that $\h=\so(1,1)$, then $b\neq 0$, $c=0$.
	 Consequently, $$\f=\big([\Real^{1,3},\Real^{1,3}]\oplus\Real e_1\wedge e_2\big)\oplus\Real^{1,3},\quad  [\Real^{1,3},\Real^{1,3}]\oplus\Real^{1,3}=\so(1,2)\oplus \h_3.$$
	 Suppose that $\h=\so(2)$, then $b=0$, $c\neq 0$, $$\f=\big([\Real^{1,3},\Real^{1,3}]\oplus\Real p\wedge q\big)\oplus\Real^{1,3},\quad  [\Real^{1,3},\Real^{1,3}]\oplus\Real^{1,3}=\h_3\oplus \f_2,$$
	 where $\f_2$ is either $\so(3)$ or $\so(1,2)$.
	 Suppose that $\h=0$, then $b=c=0$, $\g=\Real\eta$, and $\f$ contains a subalgebra transversal to $\g$ and isomorphic to the Heisenberg algebra $\h_5$.

	Suppose that $\g$ preserves the decomposition $\Real^{1,4}=\Real^{1,2}\oplus\Real^2$. 
		According to Example~\ref{ConstrdimL=3}, the torsion is given by
			$$T = \alpha p \wedge e_1 \wedge q + p \wedge e_1 \wedge v
			+p\wedge \lambda,$$
			where $v\in\Real^2$, $\lambda\in\wedge^2\Real^2$, and it holds that $\lambda(v)=0$. If $v\neq 0$, then $\lambda=0$, and
	$$T =  p \wedge e_1 \wedge (\alpha q + v).$$
			Let $v_1\in\Real^2$ be a non-zero vector orthogonal to $v$.
			It is clear that $v_1$ is annihilated by the holonomy and the torsion, i.e., the space is decomposable.
			Thus,
			$$T = \alpha p \wedge e_1 \wedge q +p\wedge \lambda,\quad \lambda=a v_1\wedge v_2,$$ where $v_1,v_2$ is an orthonormal basis of $\Real^2$. The indecomposability implies $a\neq 0$.
				\begin{align*}R(q,e_1) &= \beta p \wedge e_1 + \alpha a v_1 \wedge v_2, \\
	R(v_1, v_2) &= \alpha a p \wedge e_1 + c  v_1 \wedge v_2.  
	\end{align*}
It is easy to see that  $\f' = \f_1 \oplus \f_2$, where 
	$$ \f_1 = \left< p, e, [e, q] \right>, \quad \f_2 = \left< v_1, v_2, [v_1,v_2] \right>.$$
	The Lie algebra $\f_1$ is isomorphic to $\so(1,2)$. The Lie algebra  
	$ \f_2$ is isomorphic to $\so(1,2)$ if $c < 0$,  $ \f_2$ is isomorphic to $\so(3)$ if $c > 0$, and it is isomorphic to $\h_3$ if $c=0$.

	Suppose that $\g$ preserves the decomposition $\Real^{1,4}=\Real^{1,3}\oplus\Real^1$.
		Then as in Example~\ref{ConstrdimL>3},  
	$$ T = \alpha p \wedge e_1 \wedge e_2 + p \wedge e_1 \wedge v, $$ 
	where $v \in \mathR^1$, and $\g=p\wedge \left<e_1,e_2\right>$. In this case $(M,g)$ is a regular homogeneous plane wave.

\end{proof}

\end{theorem}

{\bf Conflict of interest statement.}
On behalf of all authors, the corresponding author states that there is no conflict of interest.

{\bf Data Availability Statements.}
Data sharing not applicable to this article as no datasets were generated or analyzed during the current study.


\begin{thebibliography}{10}
	
	
	\bibitem{AF04}  I.\,Agricola, Th.\,Friedrich, On the holonomy of connections with skew-symmetric torsion. Math. Annalen 328 (2004), no. 4, 711-748.


	\bibitem{A10}  I.\,Agricola,
Non-integrable geometries, torsion, and holonomy. in Handbook of pseudo-Riemannian Geometry and Supersymmetry,
IRMA, EMS, 2010, 277--346.

	
	\bibitem{AFe14}  I. Agricola, C. Ferreira,  Einstein manifolds with skew torsion.	Quarterly J.  Math. 65 (2014),  717--741.
	
	\bibitem{AFeF15} I. Agricola, C. Ferreira, Th. Friedrich, 	The classification of naturally reductive homogeneous spaces in dimensions $n\leq 6$.
	 Diff. Geom. Appl. 39 (2015), 59--92.
	 
	 \bibitem{BLM15} W. Batat, M. Castrillón L\'opez, E. Rosado Mar\'ia,
	 Four-dimensional naturally reductive pseudo-Riemannian spaces.
	 Diff. Geom. Appl. 41 (2015), 48--64.
	 
	
	\bibitem{Walkerbook} M. Brozos-V\'azquez, E. Garc\'ia-R\'io, P. Gilkey, S. Nik\v{c}evi\'c, R. V\'azquez-Lorenzo, The
	geometry of Walker manifolds, Synth. Lect. Math. Stat., 5, Morgan \&
	Claypool
	Publishers, Williston, VT, 2009.
	
	\bibitem{BBI} L. B\'erard-Bergery, A. Ikemakhen, On the holonomy of Lorentzian manifolds.
	 Proc. Sympos. Pure Math., 54, Amer. Math. Soc., Providence,
	RI, 1993, 27--40.
	
	
	\bibitem{Besse}  A. Besse, Einstein manifolds. Springer Verlag, 1987.
	
	\bibitem{BOL} 	M. Blau, M. O'Loughlin. Homogeneous plane waves. Nuclear Phys. B, 654 (2003), no. 1-2, 135--176, .
	
	\bibitem{CLBook} G. Calvaruso, M.C. L\'opez, Pseudo-Riemannian Homogeneous Structures. Springer 2019.
	
	\bibitem{CM08} G. Calvaruso, R.A. Marinosci, Homogeneous geodesics of non-unimodular Lorentzian Lie groups and naturally reductive Lorentzian spaces in dimension three. Adv. Geom. 8(4) (2008), 473--489.
	
	\bibitem{CMS21}
	R. Cleyton, A. Moroianu, U. Semmelmann,
	Metric connections with parallel skew-symmetric torsion. Adv. Math. 378 (2021) 107519.
	
	
	\bibitem{CS04}
	R. Cleyton, A. Swann,  Einstein metrics via intrinsic
	or parallel torsion. Math. Z. 247 (2004), 513--528. 
	
	
	\bibitem{FF09} J. Figueroa-O'Farrill, Lorentzian symmetric spaces in supergravity. In Recent Developments in Pseudo-Riemannian Geometry, ESI Lect. Math. Phys., Eur. Math. Soc., Z\"urich,  419--454.
	
	
	\bibitem{FFSPPM} {J. Figueroa-O'Farrill,  S. Philip and Patrick Meessen}, Homogeneity and plane-wave limits, J. High Energy Phys. 05(05) (2005). 
	
	
	\bibitem{FI02}  Th. Friedrich, S. Ivanov, Parallel spinors and connections with skew-symmetric torsion in string theory, Asian J. Math. 6 (2002), no. 2, 303--335.
	
	\bibitem{GalRMS}  A. S. Galaev, Holonomy groups of Lorentzian manifolds. Russian Math. Surv. 70 (2015), no. 2, 249--298.
	
	\bibitem{GL}
	W. Globke, Th.~Leistner, Locally homogeneous pp-waves. J. Geom. Phys. 108 (2016),  83--101.
	
\bibitem{HBRR}	N. Huaoari, W. Batat, N. Rahmani, S. Rahmani, Three-dimensional naturally reductive homogeneous Lorentzian manifolds. Mediterr. J. Math. 5 (2008), 113–131.
	
\bibitem{KV}	O. Kowalski and L. Vanhecke. Classification of five-dimensional naturally reductive spaces.
	Math. Proc. Cambridge Philos. Soc., 97(3):445–463, 1985.
	
	\bibitem{L07} Th.~Leistner, {On the classification of Lorentzian holonomy groups.} J. Diff. Geom. 76 (2007), no. 3, 423--484.
	
	\bibitem{LScr} Th.~Leistner, Screen bundles of Lorentzian manifolds and some generalisations of pp-waves. J. Geom. and Phys. 56 (2006), no. 10, 2117--2134.
	
	\bibitem{MSh} \'A.\,Murcia, C.S.\,Shahbazi,
	Contact metric three manifolds and Lorentzian geometry with torsion in six-dimensional supergravity.
	J.  Geom. Phys. 158 (2020),
	103868.
	
\bibitem{TV}	F. Tricerri and L. Vanhecke. Homogeneous structures on Riemannian manifolds, volume 83 of
	London Mathematical Society Lecture Note Series. Cambridge University Press, Cambridge,
	1983.
	
	\bibitem{Storm1} R. Storm, Structure theory of naturally reductive spaces. Diff. Geom. and its Appl. 64 (2019), no. 64, 174--200.
	
	\bibitem{Storm2} R. Storm, The Classification of 7- and 8-dimensional Naturally Reductive Spaces. Can. J. Math. 72 (2020), 1246--1274.
	
	
	\bibitem{Storm3} R. Storm, A new construction of
	naturally reductive spaces. Transformation Groups 23, 527--553 (2018).
	
	\bibitem{Str}  A. Strominger, Superstrings with torsion. Nucl. Phys. B274 (1986), 253--284.
	
	


\end{thebibliography}
\end{document}